\documentclass[11pt, leqno, twoside]{article}

\usepackage{amssymb}
\usepackage{amsmath}
\usepackage{amsthm}
\usepackage{amsfonts}
\usepackage{mathrsfs}
\usepackage{indentfirst}
\usepackage{graphicx}
\usepackage{txfonts}

\usepackage{anysize}

\usepackage{fancyhdr}

\usepackage[T1]{fontenc}
\usepackage[utf8]{inputenc}

\usepackage{color}
\usepackage{setspace}

\usepackage{hyperref}
\hypersetup{colorlinks=true, linkcolor=blue, citecolor=blue}

\allowdisplaybreaks

\allowdisplaybreaks

\usepackage{txfonts}

\def\rr{{\mathbb R}}
\def\rn{{{\rr}^n}}

\def\zz{{\mathbb Z}}

\def\nn{{\mathbb N}}

\def\cm{{\mathcal M}}

\def\ca{{\mathcal A}}

\def\ch{{\mathcal H}}

\def\fz{\infty}

\def\dz{\delta}

\def\lz{\lambda}

\def\lf{\left}
\def\r{\right}

\def\hs{\hspace{0.26cm}}

\def\ls{\lesssim}
\def\gs{\gtrsim}

\def\noz{\nonumber}
\def\wz{\widetilde}

\def\rh{{\mathcal {RH}}}

\def\loc{{\mathop\mathrm{\,loc\,}}}

\def\q1{\wz q}
\def\Q1{q_1}

\def\loc{{\mathop\mathrm{loc\,}}}

\newtheorem{thm}{Theorem}[section]
\newtheorem{prop}[thm]{Proposition}
\newtheorem{lem}[thm]{Lemma}
\newtheorem{cor}[thm]{Corollary}

\theoremstyle{definition}
\newtheorem{defn}[thm]{Definition}
\newtheorem{rem}[thm]{Remark}

\numberwithin{equation}{section}

\begin{document}
\title{\bf\Large Reverse H\"older Class and Two-Parameter Good-$\lambda$ Inequality over Hausdorff Content
\footnotetext{\hspace{-0.35cm}
2020 {\it Mathematics Subject Classification}.
{Primary 42B25; Secondary 28A12, 28A25.} \endgraf
{\it Key words and phrases.} Reverse H\"older class, Good-$\lambda$ inequality, Hausdorff content, capacity Muckenhoupt weight, Choquet integral. \endgraf}}
\author{Long Huang, Yangzhi Zhang\footnote{Corresponding author /{\color{red}{October 12, 2026}}}~~and Ciqiang Zhuo
}
\date{ }
\maketitle
	
\vspace{-0.7cm}

\begin{center}
\begin{minipage}{13cm}
{\small {\bf Abstract}\quad

Classical harmonic analysis is built upon Lebesgue (or more general) measure. Such a framework is remarkably successful for sets of full ambient dimension, but becomes intrinsically insensitive to lower-dimensional geometry: curves, surfaces, and more generally fractal subsets of \(\mathbb R^n\) with zero \(n\)-dimensional Lebesgue measure. In contrast,
Hausdorff content provides a dimension-sensitive notion of size for lower-dimensional sets or fractal sets, but harmonic analysis in this setting is still largely limited to maximal operators. In this paper, we establish a weighted two-parameter good-\(\lambda\) inequality under Hausdorff content setting, providing a route toward other general operators from maximal operators over Hausdorff content.
To do this, we firstly identify an appropriate  reverse H\"older class $\mathcal R\mathcal H_{q,\delta}$ associated with Hausdorff content, and then prove the self-improving property of $\mathcal R\mathcal H_{q,\delta}$ by establishing the fundamental equivalence between reverse H\"older weights and capacitary Muckenhoupt weights.
As direct applications, we obtain weighted Fefferman--Stein inequality and some boundedness on general operators over Hausdorff content.
The methodology developed in this paper, even when reduced to the Lebesgue measure case, also provide novel insights, which shows that the good-\(\lambda\) principle is not intrinsically tied to an underlying measure, but can extended to a non-additive framework.

}
\end{minipage}
\end{center}


\vspace{0.2cm}

\section{Introduction}

Given a set $E\subset \rn$ and $\dz \in (0,n]$, the $\delta$-\emph{Hausdorff content}  $\mathcal{H}_{\infty }^{\delta } (E )$ is defined by setting
\begin{align*}
\mathcal{H}_{\infty }^{\delta }\lf (E \r )
:=\inf\lf \{\sum_i [l(Q_i)]^{\delta }:\
E\subset \bigcup_i Q_i \right \},
\end{align*}
where the infimum is taken over all finite or countable cubes coverings $\{Q_i\}_{i}$ of $E$. Here and thereafter, a cube $Q\subset \mathbb{R}^n$ always denotes a left-open and right-closed cube in $\mathbb{R}^n$ with sides parallel to coordinate axes, and $l(Q)$ denotes the edge length of the cube $Q$.
The \emph{Choquet integral}  of a function $f$ with respect to Hausdorff content $\mathcal{H}_{\infty }^{\delta }$ is defined to be
\[
\int_\rn |f(x)|\,d\mathcal{H}_{\infty }^{\delta }:=\int_0^\fz \mathcal{H}_{\infty }^{\delta }(\{x\in\rn:\ |f(x)|>t\})\,dt.
\]
Based on this, for any $p\in(0,\fz)$, the \emph{$p$-Choquet integral} with respect to Hausdorff content $\mathcal{H}_{\infty }^{\delta }$ of a function $f$ on $E$ is defined by setting
\begin{align*}
\|f\|_{L^p(E,\mathcal{H}_{\infty }^{\delta })}
:=\lf[\int_E |f(x)|^p\,d\mathcal{H}_{\infty }^{\delta }\r]^\frac1p.
\end{align*}
In what follows, for any $p\in(0,\fz)$ and $E\subset \rn$, we use $L^p(E,\mathcal{H}_{\infty }^{\delta })$ to denote the space of all functions $f$ on $E$ such that the quasi-norm $\|f\|_{L^p(E,\mathcal{H}_{\infty }^{\delta })}<\fz$, and use $L^{\fz}(E,\mathcal{H}_{\infty }^{\delta })$ to denote the space of all functions $f$ on $E$ such that the quasi-norm
\[
\|f\|_{L^\infty(E,\mathcal{H}_{\infty }^{\delta })}:=\inf_{}\lf\{t>0:\mathcal{H}_{\infty }^{\delta }(\{x\in E:\ |f(x)|>t\})=0\r\}<\fz.
\]
The Choquet integral, introduced in Choquet's foundational work \cite{ch55},
has wide applications in decision theory, risk measures, and other fields; see, for example, \cite{bbc11,dhp11}. The quantity $\mathcal{H}_{\infty }^{\delta }\lf (E \r )$, also referred to as the \emph{Hausdorff content of $E$ of dimension} $\delta$, comes from fractal geometry (see \cite{eg15,fe69}) and naturally arises in nonlinear PDE because they capture the dimension-sensitive size of exceptional and singular sets that cannot be detected by Lebesgue measure (see \cite{l02}). We refer the reader to the excellent survey of Adams \cite{0923-2} for an overview on the theory of Choquet integrals over Hausdorff content.

Classical harmonic analysis is built upon Lebesgue measure. Such a framework is remarkably successful for sets of full ambient dimension, but becomes intrinsically insensitive to lower-dimensional geometry: curves, surfaces, and more generally fractal subsets of \(\mathbb R^n\), which may have zero \(n\)-dimensional Lebesgue measure and are therefore indistinguishable from the viewpoint of Lebesgue measure. In contrast, the \(\delta\)-Hausdorff content provides a finer measuring device that detects lower-dimensional and fractal structures and also retains the Hausdorff dimension through its critical exponent. This makes Hausdorff content a natural setting for developing harmonic analysis on sets that are essentially invisible to Lebesgue measure, and for formulating estimates that are sensitive not only to the ambient dimension \(n\) but also to an intrinsic dimensional parameter \(\delta\). At the same time, this passage is far from a routine replacement of \(dx\) by \(d\mathcal H^\delta_\infty\): Hausdorff content is non-additive and the associated Choquet integral is nonlinear, so fundamental tools based on countable additivity or linearity of integration are no longer available; see, for example, Adams \cite{Ad88,0923-2}.

At present, harmonic analysis in the Hausdorff content setting is still centered largely around the Hardy--Littlewood maximal operator; see, for instance, Adams \cite{0923-2}, Orobitg and Verdera \cite{0923-7}, and Chen et al. \cite{0923-3}. 
This program has developed substantially. At the same time, the weighted theory adapted to Hausdorff content has only recently begun to emerge, including Muckenhoupt weight classes $\mathcal{A}_{p,\delta}$ with Hausdorff content and reverse H\"older inequality, self-improving property, factorization theorem on $\mathcal{A}_{p,\delta}$,  as well as weighted norm inequalities for Hardy--Littlewood maximal operators \cite{HZZ25}.

By contrast, a systematic theory for cancellation-type singular integrals of Calder\'on--Zygmund type in the Hausdorff content setting, especially in the weighted setting, remains largely undeveloped.
A fundamental difficulty is that
when one attempts to move from maximal operators to singular integral operators with cancellation, in the classical theory, sparse domination has become one of the principal routes: a singular operator is first reduced to a positive sparse averaging operator, whose boundedness then yields the desired estimate. Under Hausdorff content setting, however, this second step may break down. Indeed, Hatano, Kawasumi, Saito and Tanaka \cite{hkst} showed that, for \(0<\delta<n\), the sparse operator may be unbounded on \(L^p(E,\mathcal{H}_{\infty }^{\delta })\) for every \(p>0\). Consequently, even if a singular integral admits a sparse domination, taking the Choquet norm of the sparse majorant does not in general yield a boundedness result.
This makes it necessary to develop a different route.

In classical theory, the good-\(\lambda\) inequality method provides a fundamentally different route: it converts localized control of an operator directly into distributional estimates, without requiring boundedness of a sparse dominating operator.
Thus, the aim of this paper is to establish the method of good-\(\lambda\) inequality under Hausdorff content setting.

Recall that good-$\lambda$ inequalities, introduced by Burkholder and Gundy in \cite{BDL-70}, form a very efficient and powerful tool to establish $L^p$-boundedness for various classes of operators. Let $p\in(0,\fz)$ and $w\in A_{\fz}$. For two non-negative functions $F$ and $G$, in order to prove the weighted norm inequality
\begin{align}\label{9e11}
\|F\|_{L^p(\rn,\,dw)}\le C \|G\|_{L^p(\rn,\,dw)},
\end{align}
a typical good-$\lambda$ inequality method yields that it is enough to show that for any $\alpha\in (0,1)$, there exist constants $\gamma=\gamma(\alpha)\in(0,\fz)$, $c\in(0, 1]$, and $C\in(0,\fz)$ such that for any $\lambda\in(0,\fz)$,
\[
w\Bigl(\Bigl\{x\in\rn: F(x)>2\lambda,\ G(x)\leq \gamma\lambda\Bigr\}\Bigr)\leq C\alpha^{c}w\Bigl(\Bigl\{x\in\rn: F(x)>\lambda\Bigr\}\Bigr).
\]
Furthermore, by the above good-$\lambda$ inequality, the size of $F$ is controlled by that of $G$, and hence the same inequality as \eqref{9e11}
holds with $L^{p,\fz}$ in place of $L^p$ or with other function spaces.

This technique was further developed by Coifman \cite{Coi} and Coifman--Fefferman \cite{CoFe74} to establish weighted norm inequalities between singular integral operators $T$ and Hardy--Littlewood maximal operator $M$ by setting $F:=Tf$ and $G:=Mf$
through comparing the distribution functions of the two operators. This becomes a standard tool for studying singular integral operators in classical harmonic analysis. Indeed, good-$\lambda$ inequalities also constitute a powerful tool used in probability theory and noncommutative analysis \cite{jow}.

Recently, due to the un applicable of the above good-$\lambda$ inequality for operators whose boundedness is restricted to $p\in(p_0,q_0)$, a finite range of exponents (\cite{PA07, LAC98, ZS04, ZS05}), such as the Riesz transform for the Laplace--Beltrami operator on some Riemanian manifolds \cite{PA05, PA04}, the following two-parameter good-$\lambda$ inequality was used
\begin{align}\label{eq1116-x}
\lf|\lf\{x\in\rn: M F(x)>K\lambda,\ G(x)\leq \gamma\lambda\r\}\r|\leq C\lf(\frac{1}{K^q_0}+\frac{\gamma}{K}\r)\lf|\lf\{x\in\rn: M F(x)>\lambda\r\}\r|,
\end{align}
for all $\lambda>0$, $K\ge K_0$ and $0<\gamma<1$. Subsequently, Auscher and Martell \cite{PAJMM07} established a weighted version of \eqref{eq1116-x}, namely, for $q\in(0,\fz]$, $s\in[1,\fz)$, and $w\in RH_{s'}$, under some assumptions on $F$, $G$, $H_1$, and $H_2$, there exist $C>0$ and $K_0\ge1$, such that for all $\lambda>0$, $K\ge K_0$ and $0<\gamma<1$,
\begin{align*}
&w\lf(\lf\{x\in\rn: M F(x)>K\lambda, G(x)+H_2(x)\leq \gamma\lambda\r\}\r)\\
&\quad\leq C\lf(\frac{a^q}{K^q}+\frac{\gamma}{K}\r)^{\frac1s}w\lf(\lf\{x\in\rn: M F(x)+M H_1(x)>\lambda\r\}\r).
\end{align*}
Here, $RH_q$ denotes the reverse H\"older class, introduced by Gehring in \cite{Gr73} while studying the regularity theory of quasiconformal mappings, consisting of weight functions satisfying
\[
\left( \frac{1}{|Q|} \int_Q w(x)^q\, dx \right)^{1/q} \leq C \frac{1}{|Q|} \int_Q w(x)\,dx.
\]
Gehring  \cite{Gr73} proved that $RH_q$ class possesses the self-improving property if $w \in RH_q$, then there exists $\varepsilon > 0$ such that $w \in RH_{q+\varepsilon}$. This profound property later became a key tool in proving the good-$\lambda$ inequality and studying the regularity of solutions to elliptic partial differential equations.

The main result of this paper is to obtain a Hausdorff content analogue of the above two-parameter good-$\lambda$ inequality of Auscher and Martell. Thus, we firstly invent the appropriate shape of the reverse H\"older weight class $\rh_{q,\dz}$ over Hausdorff content in Section \ref{rh01}, and prove this class also enjoying self-improving property; see Theorem \ref{lm-0921-1}.

Below is our first principal result.

\begin{thm}\label{them0923-1}
Let $\dz\in(0,n]$, $a\in [1, \fz)$, $q\in(1, \fz]$, and $w\in\rh_{s', \dz}$ with $s\in [1,\fz)$.
Let $F,\ G,\ H_1,$ and $H_2$ be non-negative functions on $\rn$ such that for any cube $Q\subset\rn$ there exist non-negative functions $G_Q$ and $H_Q$ satisfying
\begin{align}\label{eq923-1}
F(x)\leq G_Q(x)+ H_Q(x),\quad\ \text{for}\ \  \ch_\fz^\dz\text{--a. e.}\ x\in Q,
\end{align}
\begin{align}\label{eq923-2}
\biggl[\frac1{\ch_{\fz}^{\dz}(Q)}\int_{Q}H_Q(y)^q\,d\ch_{\fz}^{\dz}\biggr]^{\frac1q}
\leq a\Bigl[\cm_{\ch_{\fz}^{\dz}}F(x)+\cm_{\ch_{\fz}^{\dz}}H_1(x)+H_2(\bar{x})\Bigr],\quad \forall x,\bar{x}\in Q,
\end{align}
and
\begin{align}\label{eq923-4}
\frac1{\ch_{\fz}^{\dz}(Q)}\int_{Q}G_Q(y)\,d\ch_{\fz}^{\dz}\leq G(x),\quad\ \forall x\in Q.
\end{align}
Then there exist positive constants $C$ and $K_0> a$ such that for all $\lambda>0$, $K\ge K_0$ and $\gamma\in(0,1)$,
\begin{align}\label{eq923-5}
&w_{\ch_{\fz}^{\dz}}\lf(\lf\{x\in\rn: \cm_{\ch_{\fz}^{\dz}}F(x)>K\lambda, \, G(x)+H_2(x)\leq \gamma\lambda\r\}\r)\\
&\quad\leq C\lf(\frac{a^q}{K^q}+\frac{\gamma}{K}\r)^{\frac{1}{s}}w_{\ch_{\fz}^{\dz}}\lf(\lf\{x\in\rn: \cm_{\ch_{\fz}^{\dz}}F(x)+\cm_{\ch_{\fz}^{\dz}}H_1(x)>\lambda\r\}\r).\noz
\end{align}
Consequently, for all $p\in (0, \frac{q}{s})$, there exists a positive constant $C$ such that
\begin{align}\label{eq926-1}
\|\cm_{\ch_{\fz}^{\dz}}F\|_{L_{w}^{p}(\rn, \ch_{\fz}^{\dz})}
\leq C\Bigl(
\|G\|_{L_{w}^{p}(\rn, \ch_{\fz}^{\dz})}
+\|\cm_{\ch_{\fz}^{\dz}}H_1\|_{L_{w}^{p}(\rn, \ch_{\fz}^{\dz})}
+\|H_2\|_{L_{w}^{p}(\rn, \ch_{\fz}^{\dz})}
\Bigr),
\end{align}
provided $\|\cm_{\ch_{\fz}^{\dz}}F\|_{L_{w}^{p}(\rn, \ch_{\fz}^{\dz})}<\fz$,
and
\begin{align}\label{eq926-2}
\|\cm_{\ch_{\fz}^{\dz}}F\|_{L_{w}^{p, \fz}(\rn, \ch_{\fz}^{\dz})}
\leq C\Bigl(
\|G\|_{L_{w}^{p, \fz}(\rn, \ch_{\fz}^{\dz})}
+\|\cm_{\ch_{\fz}^{\dz}}H_1\|_{L_{w}^{p, \fz}(\rn, \ch_{\fz}^{\dz})}
+\|H_2\|_{L_{w}^{p, \fz}(\rn, \ch_{\fz}^{\dz})}
\Bigr),
\end{align}
provided $\|\cm_{\ch_{\fz}^{\dz}}F\|_{L_{w}^{p,\fz}(\rn, \ch_{\fz}^{\dz})}<\fz$.
Furthermore, if $p\ge 1$, then \eqref{eq926-1} and \eqref{eq926-2} hold
under the assumption that $F\in L^{1}(\rn, \ch_{\fz}^{\dz})$.
\end{thm}

We now give several remarks on Theorem \ref{them0923-1}.

\begin{rem}\label{rem318-1}
\begin{enumerate}
\item[{\rm(i)}] Given $s\in[1,\fz)$, we always use the standard notation $s':=\frac s{s-1}$ and a property
is said to hold \emph{$\mathcal{H}^\delta_\infty$--almost everywhere}, shortly \emph{$\mathcal{H}^\delta_\infty$--a.\,e.}, if it holds everywhere except on a set $E$ with $\mathcal{H}^\delta_\infty(E)=0$. The symbol $w_{\mathcal{H}_\infty^\delta}$ is defined as in \eqref{eq1111-3} and
$\cm_{\ch_{\fz}^{\dz}}$ denotes
uncentered Hardy--Littlewood maximal operator over Hausdorff content; see Section \ref{s11.11-1}. Weighted function spaces appeared in \eqref{eq926-1} and \eqref{eq926-2} are introduced in Section \ref{s11.11-6}.

\item[{\rm(ii)}]
When $q=\fz$, the $L^q(Q, \ch_{\fz}^{\dz})$-average appearing in \eqref{eq923-2} stand for the essential supremum $\|H_Q\|_{L^{\fz}(Q, \ch_{\fz}^{\dz})}$ and $\frac{a^q}{K^q}$ in \eqref{eq923-5} is understood as zero since $\frac a K<1$. Thus, the $L_{w}^{p}(\rn, \ch_{\fz}^{\dz})$ and $L_{w}^{p,\fz}(\rn, \ch_{\fz}^{\dz})$ estimates in \eqref{eq926-1} and \eqref{eq926-2} hold for any $p\in(0,\fz)$, no matter the value of $s$. Since $w\in\rh_{1,\dz}$ if and only if $w\in\ca_{\fz,\dz}$ by Theorem \ref{prop-0916-1} below, then, in the case $q=\fz$, Theorem \ref{them0923-1} holds for any $w\in\ca_{\fz, \dz}$.

\item[{\rm(iii)}] If \eqref{eq923-2} holds for any $q\in(1,\fz)$, then \eqref{eq926-1} and \eqref{eq926-2} hold for any $p\in(0,\fz)$ and any $w\in\ca_{\fz, \dz}$. Indeed, given $p\in(0,\fz)$
$w\in\ca_{\fz, \dz}$, by Theorem \ref{prop-0916-1}, there exists some $s\in [1,\fz)$ such that $w\in\rh_{s',\dz}$ and we may choose $q$ sufficiently large such that $p\in (0,\frac{q}{s})$.

\item[{\rm(iv)}] The self-improving property of $\rh_{s,\dz}$ in Theorem \ref{lm-0921-1} infers that
if $s\in (1,\fz)$ and $q\in(1,\fz)$, then \eqref{eq926-1} and \eqref{eq926-2} also hold for the endpoint $p=\frac{q}{s}$. To see this, notice that when $w\in\rh_{s',\dz}$, from Theorem \ref{lm-0921-1}, there exists some $s_0\in(1,s)$ such that $w\in\rh_{s_0',\dz}$, and hence, applying  Theorem \ref{them0923-1}, \eqref{eq926-1} and \eqref{eq926-2} hold for any $p\in(0,\frac{q}{s_0})$. Since $p=\frac{q}{s}\in (0,\frac{q}{s_0})$, the validity of \eqref{eq926-1} and \eqref{eq926-2} for this $p$ is immediate.

\item[{\rm(v)}] In the proof of Theorem \ref{them0923-1}, we indeed prove a stronger estimate than \eqref{eq923-5} as below
\begin{align*}
&w_{\ch_{\fz}^{\dz}}\lf(\lf\{x\in\rn: \cm_{\ch_{\fz}^{\dz}}F(x)>K\lambda, \, G(x)+H_2(x)\leq \gamma\lambda\r\}\r)\\
&\quad\leq C\lf(\frac{a^q}{K^q}+\frac{\gamma}{K}\r)^{\frac{1}{s}}w_{\ch_{\fz}^{\dz}}\lf(\lf\{x\in\rn: \cm_{\ch_{\fz}^{\dz}}F(x)>\lambda\r\}\cup\lf\{x\in\rn: \cm_{\ch_{\fz}^{\dz}}H_1(x)>\lambda\r\}\r).
\end{align*}

\item[{\rm(vi)}] When $\delta=n$, this result restricted on measurable sets and functions precisely revists the result by Auscher and Martell in \cite[Theorem 3.1]{PAJMM07}. Thus, the main contribution of Theorem \ref{them0923-1} is the two-parameter good-$\lambda$ inequality over Hausdorff content for $\delta\in(0,n)$,
which provides a flexible tool for deriving weighted Choquet integral estimates and for developing operator theory relative to lower-dimensional sets, thereby opening a route toward harmonic analysis adapted to lower-dimensional geometric size beyond maximal operators.

\item[{\rm(vii)}] Our proof strategy follows the ideas in \cite{PA07,PAJMM07}. However, the current framework introduces new challenges. For example,
the lack of countable additivity of Hausdorff content, and linearity and Fubini's theorem for Choquet integrals prevents a direct use of many standard measure-theoretic arguments. The methodology, even when reduced to the Lebesgue measure case, also provide novel insights, which shows that the good-\(\lambda\) principle, one of the central approaches in harmonic analysis, is not intrinsically tied to an underlying measure, but can survive in a non-additive framework governed by geometric size.
\end{enumerate}
\end{rem}

The proof of Theorem \ref{them0923-1} relies on the following three key ingredients, which also constitute the main difficulties. First, a direct Hausdorff content analogue of the classical reverse H\"older class of Gehring \cite{Gr73} is insufficient, since the Choquet integral is nonlinear and Hausdorff content is non-additive. We therefore introduce a suitable class of reverse H\"older weights \(\rh_{q,\dz}\) over Hausdorff content in Section \ref{rh01}. This definition is intrinsic in the present setting, as it is in fact equivalent to the corresponding reverse H\"older inequality in the Hausdorff content framework; see Remark \ref{09re13}. In particular, it guarantees \eqref{eq0902-3} and \eqref{eq0902-2}, which make the subsequent analysis possible. Second, many classical arguments for reverse H\"older class rely essentially on linearity of integration and hence break down in the present setting. To overcome this obstruction, we establish an equivalence between reverse H\"older weights \(\rh_{1,\dz}\)  and capacitary Muckenhoupt weights \(\ca_{\fz,\dz}\); see Theorem \ref{prop-0916-1} and Lemma \ref{lm0922-1}. This reduces the problem to the theory of capacitary Muckenhoupt weights, where many useful tools developed in \cite{HZZ25} become available.
Third, the self-improving property of classical \(RH_q\) weights is a crucial ingredient for proving good-\(\lambda\) inequalities. Its counterpart over Hausdorff content is, however, substantially more delicate, as it requires establishing an appropriate reverse H\"older inequality over weighted capacitary \(w_{\ch_{\fz}^{\dz}}\).


The organization of this paper is as follows. Section \ref{s2} reviews some necessary basic concepts, including  capacity, Choquet integrals, and capacitary Muckenhoupt $\ca_{p,\dz}$ weights as well as some required properties. Section \ref{rh01} introduces capacitary reverse H\"older class $\rh_{q, \dz}$ and
establishes fundamental relationships between capacitary Muckenhoupt $\ca_{p,\dz}$ weights and capacitary reverse H\"older class $\rh_{q, \dz}$. The reverse H\"older inequality over $w_{\ch_{\fz}^{\dz}}$ and self-improving propery of $\rh_{q, \dz}$ is also established in this section. Section \ref{s4} proves Theorem \ref{them0923-1}. Section \ref{s5} deduces some applications of Theorem \ref{them0923-1}, such as weighted Fefferman--Stein inequality and some singular non-integral operators over Hausdorff content.

Throughout the paper, let $\delta\in(0,n]$. The notation $f\ls g$ (resp. $f\gs g$) means $f\le Kg$ (resp. $f\ge Kg$)
for a positive constant $K$ independent of the main parameters, and
$f\sim g$ amounts to $f\ls g\ls f$. We also use the symbol $K(\alpha,\,\beta,\dots)$ to denote a positive constant
which depends on the parameters $\alpha,\,\beta,\dots$ but may vary line to line.
Also, we denote by $\mathbf{1}_E$ the characteristic function of set $E\subset \rn$.
For any $x\in\rn$ and $r\in(0,\fz)$, $Q(x,r)$ stands for the cube centered at $x$ with side length $r$. Given a cube $Q\subset \rn$ and $\lambda\in(0,\fz)$, the new cube $\lambda Q$ always denotes the cube that shares the same center as $Q$
and whose side length is scaled by a factor of $\lambda$.

\section{Choquet Integral, Capacitary Maximal Operator, and $\ca_{p,\dz}$ Weight}\label{s2}

\subsection{Capacity and Choquet Integrals}\label{s11.11}
A set function $C$ is called a \emph{capacity} if it is a non-negative, monotone and countably sub-additive set function,
i.e., $C$ satisfies:
\begin{enumerate}
\item[(i)] $C(\emptyset)=0$ and, for any set $E\subset \mathbb R^n$, $C(E)\ge0$;
\item[(ii)] If $E\subset F$, then $C(E)\le C(F)$;
\item[(iii)] If $E=\bigcup_{i=1}^{\fz}E_i$, then
$C(E)\le \sum_{i=1}^{\fz}C(E_i)$.
\end{enumerate}

Let $p\in(0, \fz]$, $E\subset \rn$ and $C$ be a capacity. For any $C$-almost everywhere defined function $f$ on $E$, its \emph{Choquet integral} with respect to capacity $C$ is defined to be
\[
\int_\rn |f(x)|\,dC:=\int_0^\fz C(\{x\in\rn:\ |f(x)|>t\})\,dt.
\]
Based on this, for any $p\in(0,\fz)$, the \emph{$p$-Choquet integral} with respect to capacity $C$ of a function $f$ on $E$ is defined by setting
\begin{align*}
\|f\|_{L^p(E,C)}
:=\lf[\int_E |f(x)|^p\,dC\r]^\frac1p.
\end{align*}
In what follows, for any $p\in(0,\fz)$ and $E\subset \rn$, we use $L^p(E,C)$ to denote the space of all functions $f$ on $E$ such that the quasi-norm $\|f\|_{L^p(E,C)}<\fz$, and use $L^{\fz}(E,C)$ to denote the space of all functions $f$ on $E$ such that the quasi-norm
\[
\|f\|_{L^\infty(E,C)}:=\inf_{t>0}\lf\{t:C(\{x\in E:\ |f(x)|>t\})=0\r\}<\fz.
\]

Below we give several needed remarks on capacity, which will be frequently used in this paper.
\begin{rem}\label{rem-11112-1}
Let $C$ be a capacity.
\begin{enumerate}
\item[(i)]
Denote by $L_{\loc}^1(\rn,C)$ the set of all functions satisfying that, for any compact set $E\subset\rn$, $f\in L^1(E,C)$. It is straightforward to observe that, for any sequence $\{E_j\}_{j\in\nn}$ of subsets in $\rn$,
\begin{align}\label{eq1111-5}
\int_{\bigcup_{j\in\nn}E_j}|f(x)|\,dC\le \sum_{j\in\nn}\int_{E_j}|f(x)|\,dC.
\end{align}
Consequently, if $f\in L_{\loc}^1(\rn,C)$ is non-negative, then the set function $f_C$ is also a capacity,
where $f_C$ is defined by setting, for any subset $E\subset \rn$, 
\begin{align}\label{eq1111-6}
f_C(E):=\int_{E}f(x)\,dC.
\end{align}
Clearly, for any $\delta\in(0,n]$, the Hausdorff content $\ch_{\fz}^{\dz}$ is a capacity; see, for instance, \cite{ch55}.
Therefore, if $w$ is a non-negative function on $\rn$ belonging to $L_{\loc}^1(\rn,\ch_{\fz}^{\dz})$,
then the set function $w_{\ch_{\fz}^{\dz}}$ is also a capacity.
Furthermore, if $F$ is a non-negative function on $\rn$ such that $F\in L_{\loc}^1(\rn,w_{\ch_{\fz}^{\dz}})$,
then $F_{w_{\ch_{\fz}^{\dz}}}$ is a capacity as well.
\item[(ii)]
(H\"{o}lder's inequality) Let $p\in(1,\infty)$.
Then, for any functions $f$ and $g$ on $\rn$, we have
    \[\int_{\rn}|f(x)g(x)|\,dC\leq 2\lf(\int_{\rn}|f(x)|^p\,dC\r)^{\frac{1}{p}}\lf(\int_{\rn}|g(x)|^q\,dC\r)^{\frac{1}{q}},\]
where $\frac1p+\frac1q=1$.
\item[(iii)] Let $N\in \nn$, $E\subset \rn$, and $\{f_j\}_{j=1}^N$ be a sequence of functions defined on $E$. Then
\[\int_E\lf|\sum_{j=1}^Nf_j(x)\r|\,dC\leq N\sum_{j=1}^{N}\int_E\lf|f_j(x)\r|\,dC.\]
\item[(iv)] Let $K$ be a positive constant and $E\subset \rn$ with $C(E)<\fz$. Then, for any function $f$ on $E$,
\[\int_{E}\lf[|f(x)|+K\r]\,dC=\int_{E}|f(x)|\,dC+KC(E).\]
\end{enumerate}
\end{rem}
Let $C$ be a capacity. For any non-negative function $f$ and any set $E\subset \rn$, we use the notation
\[\fint_Ef(x)\,d C:=\frac1{C(E)}\int_Ef(x)\,d C.\]
In this paper, we mainly focus on the Choquet integral associated with Hausdorff content. In the following, we summarize some required results regarding the Hausdorff content.
\begin{rem}\label{12-4}
Let $\delta\in(0,n]$.

\begin{enumerate}
\item[(i)] For any cube $Q$ of $\rn$, we have $\ch^{\delta}_{\fz}(Q)=[l(Q)]^{\delta}$ and, for any ball $B(x,r)$ with $x\in \rn$ and $r\in (0,\fz)$,
$\ch^{\delta}_{\fz}(B(x, r))\sim r^{\delta}$.

\item[(ii)] Let $\widetilde{\ch}_\fz^{\delta}$ be the \emph{dyadic Hausdorff content}, which is defined by the same way as
 $\ch^{\delta}_{\fz}$ but with cubes coverings $\{Q_i\}_i$ replaced by dyadic cubes coverings.
 Then there exists a positive constant $K(n, \delta)$ such that, for any subset $E\subset \rn$, $\ch^{\delta}_{\fz}(E)\le \widetilde{\ch}_\fz^{\delta}(E)\le K(n, \delta)\ch^{\delta}_{\fz}(E).$
 This equivalent Hausdorff content was proved to be a capacity in the sense of Choquet; see, for example, D. R. Adams \cite{Ad88}.

\item[(iii)] Applying the equivalence of $\ch_\fz^\delta$ and $\widetilde{\ch}_\fz^\delta$ and the sublinearity
of the Choquet integral with respect to $\widetilde{\ch}_\fz^\delta$ (see \cite[p.\,13, Theorem 1]{0923-2}),
we conclude that there exists a positive constant $K(n, \delta)$ such that, for any non-negative functions $\{f_j\}_{j\in\nn}$ defined on
$E$,
$$\int_{E}\sum_{j\in\nn}f_j(x)\,d\ch_\fz^\delta
\le K(n, \delta)\sum_{j\in\nn}\int_{E}f_j(x)\,d\ch_\fz^\delta.$$
\item[(iv)]
Let $E\subset \rn$ and $\{f_n\}_{n\in\nn}$ be a sequence of non-negative functions defined on $E$. If $f_{n}\uparrow f$, then
\[\lim_{n\to\fz}\int_{E}f_n(x)\,d\widetilde{\ch}_\fz^{\delta}=\int_{E}f(x)\,d\widetilde{\ch}_\fz^{\delta}.\]
\end{enumerate}
\end{rem}

\subsection{Capacitary Maximal Operator}\label{s11.11-1}
For $f \in L_{\loc}^1(\rn,\ch_{\fz}^{\dz})$, the \emph{capacitary uncentered Hardy--Littlewood maximal function} $\cm_{\ch_{\fz}^{\dz}}f$ of $f$ is defined by
\begin{equation}\label{eq319-2}
\cm_{\ch_{\fz}^{\dz}}f(x):=\sup_{Q\ni x}\fint_Q|f(y)|\,d\ch_{\fz}^{\dz},\quad \forall\,x\in\rn,
\end{equation}
where the supremum is taken over all cubes $Q\subset\rn$ containing $x$.

\begin{rem}\label{rem925-1}
\begin{enumerate}
\item[{\rm(i)}]
Let $\widetilde\cm_{\ch_{\fz}^{\dz}}$ be defined in the same way as $\cm_{\ch_{\fz}^{\dz}}$, but with cubes replaced by open cubes. Then, for any $f\in L^1_{\loc}(\rn, \ch_{\fz}^{\dz})$,
\begin{align}\label{eq924-1}
\widetilde\cm_{\ch_{\fz}^{\dz}}f(x)=\cm_{\ch_{\fz}^{\dz}}f(x),\quad \forall\,x\in\rn.
\end{align}
Indeed, for any $f\in L^1_{\loc}(\rn, \ch_{\fz}^{\dz})$ and any $x\in\rn$, it is clear that
$\widetilde\cm_{\ch_{\fz}^{\dz}}f(x)\le\cm_{\ch_{\fz}^{\dz}}f(x).$
Furthermore, given $x\in\rn$, by the definition of $\cm_{\ch_{\fz}^{\dz}}$, for any $\varepsilon>0$, there exists a cube $Q(x_0,l)$ containing $x$ such that
\[\cm_{\ch_{\fz}^{\dz}}f(x)-\varepsilon\leq \fint_{Q(x_0,l)}|f(y)|\,d\ch_{\fz}^{\dz}.\]
For any $\eta>0$, let $Q_1$ be the open cube centered at $x_0$ with side length $l+\eta$. Then $Q(x_0,l)\subset Q_1$, and hence
from Remark \ref{12-4}(i), it follows that
\[\cm_{\ch_{\fz}^{\dz}}f(x)-\varepsilon\leq \frac{(l+\eta)^{\dz}}{l^{\dz}}\fint_{Q_1}|f(y)|\,d\ch_{\fz}^{\dz}\leq\frac{(l+\eta)^{\dz}}{l^{\dz}}\widetilde\cm_{\ch_{\fz}^{\dz}}f(x).\]
Finally, letting $\eta\to 0$ first and then $\varepsilon\to 0$, we conclude that, for any $x\in \rn$,
$$\cm_{\ch_{\fz}^{\dz}}f(x)\leq \widetilde\cm_{\ch_{\fz}^{\dz}}f(x).$$
Thus \eqref{eq924-1} holds.
\item[{\rm(ii)}]
For any $f\in L^1_{\loc}(\rn, \ch_{\fz}^{\dz})$ and any $a>0$, the set $\{x\in\rn: \cm_{\ch_{\fz}^{\dz}}f(x)>a\}$ is open.
Indeed, by (i), it suffices to show that $F_a:=\{x\in\rn: \widetilde\cm_{\ch_{\fz}^{\dz}}f(x)>a\}$ is open.
To this end, we may assume without loss of generality that $F_a\neq \emptyset$.
For any $x\in F_a$, there exists an open cube $Q$ containing $x$ such that
\[\fint_{Q}|f(y)|\,d\ch_{\fz}^{\dz}>a.\]
Then $Q\subset F_a$. Thus there exists $\delta>0$ such that $B(x,\delta)\subset Q\subset F_a$ and hence $F_a$ is open.

\item[{\rm(iii)}]
Let $\cm^*_{\ch_{\fz}^{\dz}}$ be the \emph{capacitary centered Hardy-Littlewood maximal function},
which is defined by setting, for any $f\in L_{\rm loc}^1(\rn,\ch_\fz^\delta)$,
\begin{equation}\label{eq319-2}
\cm_{\ch_{\fz}^{\dz}}f(x):=\sup_{Q\ni x}\fint_Q|f(y)|\,d\ch_{\fz}^{\dz},\quad \forall\,x\in\rn,
\end{equation}
where the supremum is taken over all cubes $Q\subset\rn$ centered at $x$.
Then there exists a positive constant $C_0$ independent of $f$ such that, for any $x\in \rn$,
$$\frac 1{C_0}\cm^*_{\ch_{\fz}^{\dz}}f(x)\le \cm_{\ch_{\fz}^{\dz}}f(x)\le  C_0 \cm^*_{\ch_{\fz}^{\dz}}f(x)$$.

\item[{\rm(iv)}]
By Remark \ref{rem-11112-1}(iii), we find that, for any $f,g\in L^1_{\loc}(\rn, \ch_{\fz}^{\dz})$,
\[
\cm_{\ch_{\fz}^{\dz}}(f+g)\leq 2 \lf(\cm_{\ch_{\fz}^{\dz}}f+ \cm_{\ch_{\fz}^{\dz}}g\r).
\]
Furthermore, it follows from Remark \ref{12-4}(iii) that,
for any sequence of non-negative functions $\{f_j\}_{j\in\nn}$ and any sequence of non-negative numbers $\{c_j\}_{j\in\nn}$,
\[
\cm_{\ch_{\fz}^{\dz}}\left(\sum_{j\in\nn}c_jf_{j}\right)
\ls \sum_{j\in\nn}c_j\cm_{\ch_{\fz}^{\dz}}(f_{j}).
\]
\item[{\rm(v)}]
Let $p\in[1, \fz)$ and $f\in L^p(\rn, \ch_{\fz}^{\dz})$. Then, by \cite[Theorem 1.2]{ChSp2024}, we have
\[
\ch_{\fz}^{\dz}\left(\left\{x\in\rn: \cm_{\ch_{\fz}^{\dz}}f(x)>t\right\}\right)
\ls \frac{1}{t^p}\int_{\rn}|f(x)|^p\,d\ch_{\fz}^{\dz},\quad \forall\,t\in(0,\fz).
\]

\item[{\rm(vi)}]
It is clear that for any $f\in L^{\fz}(\rn, \ch_{\fz}^{\dz})$,
$
\|\cm_{\ch_{\fz}^{\dz}}f\|_{L^{\fz}(\rn, \ch_{\fz}^{\dz})}\leq \|f\|_{L^{\fz}(\rn, \ch_{\fz}^{\dz})}.
$
\end{enumerate}
\end{rem}

\subsection{Capacitary Muckenhoupt Class $\ca_{p,\dz}$}\label{s11.11-2}

In this subsection, let us first recall the \emph{capacitary Muckenhoupt $\mathcal A_{p,\delta}$-weight} introduced in \cite{HZZ25}. A function $w$ is called a $\ch_\fz^\delta$-\emph{capacitary weight} on $\rn$
if $w(x)\in(0,\infty)$  $\mathcal{H}_\infty^\delta$--a.e. and $w \in L_{\text{loc}}^1(\mathbb{R}^n, \mathcal{H}_\infty^\delta)$. For $p \in (1, \infty)$, a $\ch_\fz^\delta$-capacitary weight $w$ is said to be in the $ \mathcal{A}_{p,\delta}$ class
if
\begin{align}\label{eq1113-1}
[w]_{\mathcal A_{p,\delta}} := \sup_{\substack{\text{cube } Q \subset \mathbb{R}^n}} \left\{ \fint_Q w(x) \, d\mathcal{H}_\infty^\delta \right\} \left\{ \fint_Q w(x)^{-\frac{1}{p-1}} \, d\mathcal{H}_\infty^\delta \right\}^{p-1} < \infty.
\end{align}
A $\ch_\fz^\delta$-capacitary weight $w$ is said to be in the $\mathcal{A}_{1,\delta}$ class if 
\begin{align}\label{eq1113-2}
[w]_{\mathcal{A}_{1,\delta}} := \sup_{\substack{\text{cube } Q \subset \mathbb{R}^n}} \left\{ \fint_Q w(x) \, d\mathcal{H}_\infty^\delta \right\} \| w^{-1} \|_{L^\infty(Q, \mathcal{H}_\infty^\delta)} < \infty.
\end{align}
Moreover, $\mathcal{A}_{\infty,\delta} := \bigcup_{p \in [1, \infty)} \mathcal{A}_{p,\delta}$.

Furthermore, let $p \in [1, \infty)$ and $w$ be a $\ch_\fz^\delta$-capacitary weight. 
Then it follows from Remark \ref{rem-11112-1}(i) that $w_{\mathcal{H}_\infty^\delta}$ defined as \eqref{eq1111-6}, i.e.
\begin{equation}\label{eq1111-3}
w _{\mathcal H^{\delta}_\infty}(E):=\int_Ew(x)\,d\ch_\fz^\delta,\quad \forall\,E\subset \rn,
\end{equation}
 is also a capacity.
Suppose $F \in L_{\text{loc}}^1(\mathbb{R}^n, w_{\mathcal{H}_\infty^\delta})$ and $F(x)\in(0, \infty)$ $w_{\mathcal{H}_\infty^\delta}$--a.e.. For $p \in (1, \infty)$, we say $F \in \mathcal{A}_{p, w_{\mathcal{H}_\infty^\delta}}$ if the inequality \eqref{eq1113-1} still holds for $w$ replaced by $F$ and $\mathcal{H}_\infty^\delta$ replaced by $w_{\mathcal{H}_\infty^\delta}$. The corresponding capacitary weight constant is denoted $[F]_{\mathcal{A}_{p, w_{\mathcal{H}_\infty^\delta}}}$. Similarly, $F \in \mathcal{A}_{1, w_{\mathcal{H}_\infty^\delta}}$ if the inequality \eqref{eq1113-2} still holds for $w$ replaced by $F$ and $\mathcal{H}_\infty^\delta$ replaced by $w_{\mathcal{H}_\infty^\delta}$, and the corresponding constant is denoted $[F]_{\mathcal{A}_{1, w_{\mathcal{H}_\infty^\delta}}}$.

We summarize some properties of $\ca_{p,\dz}$ in the following proposition.

\begin{prop}\label{prop-0901-1}
Let $\delta\in(0,n]$.

\begin{enumerate}
\item[\rm{(i)}]
If $1<p<q<\fz$, then $\ca_{1,\dz}\subsetneqq \ca_{p,\dz}\subsetneqq \ca_{q,\dz}\subsetneqq \mathcal{A}_{\infty,\delta}$.

\item[\rm{(ii)}]
Let $p\in[1,\fz)$ and $w\in\ca_{p,\dz}$. Then there exist positive constant $K$ such that, for any cube $Q\subset\rn$,
    \begin{align}\label{eq1111-1}
    w_{\ch_{\fz}^{\dz}}(2Q)\leq Kw_{\ch_{\fz}^{\dz}}(Q)
    \end{align}
    and
   \begin{align}\label{eq1111-2}
   \sum_{j\in\zz}\ch_{\fz}^{\dz}(  W_j\cap Q)\leq K\ch_{\fz}^{\dz}(Q).
   \end{align}
Here, for any $j\in\zz$,
$
  W_j:=\{x\in \rn: 2^{j-1}\leq w(x)< 2^j\}.
$

\item[\rm{(iii)}]
If $p\in(1,\fz)$, then $w\in\ca_{p,\dz}$ if and only if
$w^{-\frac{1}{p-1}}\in\ca_{\frac{p}{p-1},\dz}$.

\item[\rm{(iv)}]
Let $p\in [1,\fz)$ and $w\in\ca_{p,\dz}$. Then there exist positive constants $K$ and $\gamma\in(0,1)$ such that,
for any cube $Q$,
\begin{align*}
\lf\{\fint_{Q}w(x)^{1+\gamma}\,d\ch_{\infty}^{\delta}\r\}^{\frac{1}{1+\gamma}}\leq K\fint_{Q}w(x)\,d\ch_{\infty}^{\delta}.
\end{align*}

\item[\rm{(v)}]
Let $p\in(1,\infty)$ and $w\in \ca_{p,\delta}$. Then there exists $q$ with $1<q<p$ such that $w\in \ca_{q,\delta}$.

\item[\rm{(vi)}]
Let $p\in [1,\fz)$, $w\in\ca_{p,\dz}$ and $\gamma$ be as in {\rm (iv)}. 
Then there exists a constant $\gamma_1\in(0, \gamma]$ such that $w^{1+\gamma_1}\in \ca_{p,\delta}$.
\end{enumerate}
\end{prop}

\begin{proof}
Properties (i)--(v) are standard and can be found in \cite{HZZ25}.
We now prove (vi).

Let $p\in[1,\fz)$ and $w\in\ca_{1,\dz}$. Then, by (iv), there exists a positive constant $\gamma$ such that, for any cube $Q$,
\[
\fint_{Q}w(y)^{1+\gamma}\,d\ch_{\infty}^{\delta}\ls \lf\{\fint_{Q}w(y)\,d\ch_{\infty}^{\delta}\r\}^{1+\gamma}\ls w(x)^{1+\gamma}
\]
for $\ch_\fz^\delta$--a.e. $x\in Q$. When $p=1$, by choosing $\gamma_1=\gamma$, we hace $w^{1+\gamma_1}\in\ca_{1,\dz}$.

When $p\in(1,\fz)$,  then $w^{-\frac1{p-1}}\in \mathcal A_{\frac p{p-1},\delta}$ and hence, by (iv),
there exist a positive constant $\widetilde{\gamma}$ such that, for every cube $Q$,
\begin{align*}
\lf\{\fint_{Q}w(x)^{-\frac{1+\widetilde{\gamma}}{p-1}}\,d\ch_{\infty}^{\delta}\r\}^{\frac{1}{1+\widetilde{\gamma}}}\ls \fint_{Q}w(x)^{-\frac{1}{p-1}}\,d\ch_{\infty}^{\delta}.
\end{align*}
Set $\gamma_1=\min\{\gamma, \widetilde{\gamma}\}$. By H\"older's inequality and $w\in\ca_{p,\dz}$, for any cube $Q\subset\rn$, we have
\begin{align*}
&\lf\{\fint_Qw(x)^{1+\gamma_1}\,d\mathcal H^{\delta}_\infty\r\}\lf\{\fint_Qw(x)^{-\frac{1+\gamma_1}{p-1}}\,d\mathcal H^{\delta}_\infty\r\}^{p-1}\\
&\quad\ls\lf\{\fint_{Q}w(x)^{1+\gamma}\,d\ch_{\infty}^{\delta}\r\}^{\frac{1+\gamma_1}{1+\gamma}}\lf\{\fint_Qw(x)^{-\frac{1+\widetilde{\gamma}}{p-1}}\,d\mathcal H^{\delta}_\infty\r\}^{\frac{(p-1)(1+\gamma_1)}{1+\widetilde{\gamma}}}\\
&\quad\ls\lf\{\fint_{Q}w(x)\,d\ch_{\infty}^{\delta}\r\}^{1+\gamma_1}\lf\{\fint_{Q}w(x)^{-\frac{1}{p-1}}\,d\ch_{\infty}^{\delta}\r\}^{(p-1)(1+\gamma_1)}\ls 1,
\end{align*}
which implies that $w^{1+\gamma_1}\in\ca_{p,\dz}$.
This completes the proof of (vi).
\end{proof}

\section{Reverse H\"older Class $\rh_{q,\dz}$ over Hausdorff Content}\label{rh01}

In this section, we introduce the reverse H\"older class $\rh_{q,\dz}$ over Hausdorff content as follows.

\begin{defn}\label{rh}
Let $\delta\in(0,n]$. The \emph{capacitary reverse H\"older class $\rh_{q,\dz}$} with $q\in(1,\fz)$ is defined to be the set of all $\ch_\fz^\delta$-capacitary weights $w$ satisfying
\[
[w]_{\rh_{q,\dz}}:=\sup_{\text{cube}\,Q\,\subset \rn}\frac{1}{w_{\ch_{\fz}^{\dz}}(Q)}\left\{\sum_{j\in\zz}\ch_{\fz}^{\dz}(   W_j\cap Q)\right\}^{1-\frac{1}{q}}\left\{\int_{2Q}w(x)^q\,d\ch_{\fz}^{\dz}\right\}^{\frac{1}{q}}<\fz.
\]
Here and in what follows, for any $j\in\zz$, $$   W_j:=\{x\in\rn: 2^{j-1}<w(x)\le 2^j\}.$$
The \emph{capacitary reverse H\"older class $\rh_{\infty,\dz}$} is defined to be the set of all $\ch_\fz^\delta$-capacitary weight $w$ satisfying
\[
[w]_{\rh_{\fz,\dz}}:=\sup_{\text{cube}\,Q\,\subset \rn}\frac{1}{w_{\ch_{\fz}^{\dz}}(Q)}\left\{\sum_{j\in\zz}\ch_{\fz}^{\dz}(   W_j\cap Q)\|w\|_{L^{\fz}(2Q, \ch_{\fz}^{\dz})}\right\}<\fz.
\]
Set $\rh_{1,\dz}:=\bigcup_{q\in(1,\fz]}\rh_{q,\dz}.$
\end{defn}

At the special case $\delta=n$, the capacitary reverse H\"older class $\rh_{q,\dz}$, restricted on measurable sets and functions, coincides with the classical reverse H\"older class $RH_q$ of Gehring \cite{Gr73}.

Concerning this definition, we give some remarks below.

\begin{rem}
Let $w$ be a $\ch_\fz^\delta$-capacitary weight. Then, for any subset $E\subset \rn$,
\begin{align}\label{eq1113-5}
\ch_{\fz}^{\dz}(E)\leq \sum_{j\in\zz}\ch_{\fz}^{\dz}(    W_j\cap E).
\end{align}
Indeed, this is a consequence of the decomposition
\[E=\bigcup_{j\in\zz}(   W_j\cap E)\cup \{x\in\rn: w(x)=0\}\cup \{x\in\rn: w(x)=\fz\}\]
 and the countable subadditivity of $\ch_{\fz}^{\dz}$.

\end{rem}

\begin{rem}\label{re109}
Let $w\in\rh_{q,\dz}$ with $q\in(1,\fz)$. Then the following reverse H\"older inequality
holds
\begin{align}\label{8e26}\lf(\fint_Qw(x)^qd\ch_{\fz}^{\dz}\r)^{\frac1q}\le [w]_{\rh_{q,\dz}} \fint_Qw(x)d\ch_{\fz}^{\dz},\quad\forall\, \text{cube} \ Q\subset\rn.
\end{align}
Indeed, by \eqref{eq1113-5}, we find that, for any $Q\subset\rn$,
\[
\ch_{\fz}^{\dz}(Q)^{1-\frac{1}{q}}\left(\int_{Q}w(x)^q\,d\ch_{\fz}^{\dz}\right)^{\frac{1}{q}}\leq\lf[\sum_{j\in\zz}\ch_{\fz}^{\dz}(   W_j\cap Q)\r]^{1-\frac{1}{q}}\left(\int_{2Q}w(x)^q\,d\ch_{\fz}^{\dz}\right)^{\frac{1}{q}}
\le [w]_{\rh_{q,\dz}}w_{\ch_{\infty}^{\delta}}(Q),
\]
which implies \eqref{8e26}.
Similarly, if $w\in\rh_{\infty,\dz}$. then, for any cube $Q\subset \rn$,
\begin{align}\label{8e30}
\|w\|_{L^{\fz}(Q, \ch_{\fz}^{\dz})}\le[w]_{\rh_{\fz,\dz}}\fint_{Q}w(x)d\ch_{\fz}^{\dz}.\end{align}
\end{rem}

\subsection{Basic Properties of
$\rh_{q,\dz}$}\label{s11.11-6}

In this subsection we prove the following basic properties on the newly introduced capacitary reverse H\"older class $\rh_{q,\dz}$.

\begin{prop}\label{prop-0902-1}
\begin{enumerate}
\item[\rm{(i)}]
If $1<p<q<\fz$, then $\rh_{\fz,\dz}\subset\rh_{q,\dz}\subset\rh_{p,\dz}\subset\rh_{1,\dz}$.

\item[\rm{(ii)}]
Let $w\in\rh_{p,\dz}$ with $p\in (1,\fz]$. Then there exists a positive constant $C$ such that for any cube $Q\subset\rn$,
\begin{equation}\label{eq0902-3}
w_{\ch_{\fz}^{\dz}}(2Q)\leq Cw_{\ch_{\fz}^{\dz}}(Q)
\end{equation}
and
\begin{equation}\label{eq0902-2}
\sum_{j\in\zz}\ch_{\fz}^{\dz}(   W_j\cap Q)\leq C\ch_{\fz}^{\dz}(Q).
\end{equation}
\end{enumerate}
\end{prop}

\begin{proof}
(i) Let \(w\in \rh_{\infty,\delta}\). Then for any cube \(Q\subset\mathbb R^n\) and for \(\ch^\delta_\infty\)--a. e. \(x\in 2Q\),
\[
\sum_{j\in\zz}\ch_{\fz}^{\dz}(   W_j\cap Q)w(x)\le [w]_{\rh_{\fz,\dz}}w_{\ch_{\fz}^{\dz}}(Q).
\]
By this, $\ch_{\fz}^{\dz}(2Q)=2^{\dz}\ch_{\fz}^{\dz}(Q)$ and \eqref{eq1113-5}, 
we obtain, for any cube \(Q\subset\mathbb R^n\),
\begin{align*}
\sum_{j\in\zz}\ch_{\fz}^{\dz}(   W_j\cap Q)\left\{\int_{2Q}w(x)^q\,d\ch_{\fz}^{\dz}\right\}^{\frac{1}{q}}
&\le [w]_{\rh_{\fz,\dz}}
\ch_{\fz}^{\dz}(2Q)^{\frac1q}w_{\ch_{\fz}^{\dz}}(Q)\\
&\le2^{\frac{\dz}{q}}[w]_{\rh_{\fz,\dz}} \left\{\sum_{j\in\zz}\ch_{\fz}^{\dz}(  W_j\cap Q)\right\}^{\frac1q}w_{\ch_{\fz}^{\dz}}(Q).
\end{align*}
This implies that \([w]_{\rh_{q,\dz}}
\le2^{\frac{\dz}{q}}[w]_{\rh_{\fz,\dz}}\)
and hence $\rh_{\fz,\dz}\subset\rh_{q,\dz}$.

Next, we show that $\rh_{q,\dz}\subset\rh_{p,\dz}$. For any $w\in \rh_{q,\dz}$, by H\"older's inequality, $\ch_{\fz}^{\dz}(2Q)=2^{\dz}\ch_{\fz}^{\dz}(Q)$ and \eqref{eq1113-5},  we have, for any cube \(Q\subset\mathbb R^n\),
\[
\begin{aligned}
&\left\{\sum_{j\in\zz}\ch_{\fz}^{\dz}(  W_j\cap Q)\right\}^{1-\frac{1}{p}}\left\{\int_{2Q}w(x)^p\,d\ch_{\fz}^{\dz}\right\}^{\frac{1}{p}}\\
&\quad\le2\left\{\sum_{j\in\zz}\ch_{\fz}^{\dz}(  W_j\cap Q)\right\}^{1-\frac{1}{q}}\left\{\int_{2Q}w(x)^q\,d\ch_{\fz}^{\dz}\right\}^{\frac{1}{q}}\ch_{\fz}^{\dz}(2Q)^{\frac1p-\frac1q}\left\{\sum_{j\in\zz}\ch_{\fz}^{\dz}(  W_j\cap Q)\right\}^{\frac1q-\frac{1}{p}}\\
&\quad \le 2^{\dz(\frac1p-\frac1q)+1}[w]_{\rh_{q,\dz}}
w_{\ch_{\fz}^{\dz}}(Q),
\end{aligned}
\]
which implies
\[
[w]_{\rh_{p,\dz}}
\le 2^{\dz(\frac1p-\frac1q)+1}[w]_{\rh_{q,\dz}}.
\]
Thus, $\rh_{q,\dz}\subset\rh_{p,\dz}$. It is obvious that $\rh_{p,\dz}\subset\rh_{1,\dz}$. Therefore, (i) is proved.

(ii) We first claim that, if $w\in \mathcal R\mathcal H_{p,\delta}$ with $p\in(1,\fz]$,
then for any cube $Q$ of $\rn$ and $E\subset 2Q$,
\begin{equation}\label{eq0902-1}
\frac{w_{\ch_{\fz}^{\dz}}(E)}{w_{\ch_{\fz}^{\dz}}(Q)}\le 2[w]_{\rh_{p,\dz}} \left\{\frac{\ch_{\fz}^{\dz}(E)}{\sum_{j\in\zz}\ch_{\fz}^{\dz}(  W_j\cap Q)}\right\}^{1-\frac1p}.
\end{equation}
If this claim is true, taking $E=Q$ in \eqref{eq0902-1} immediately yields \eqref{eq0902-2}, and taking $E=2Q$ in \eqref{eq0902-1}, together with the fact $\ch_{\fz}^{\dz}(2Q)=2^{\dz}\ch_{\fz}^{\dz}(Q)$ and \eqref{eq1113-5}, immediately implies \eqref{eq0902-3}.

We now only need to show \eqref{eq0902-1}. If $p=\fz$, namely, $w\in\rh_{\fz,\dz}$,
then for any cube $Q\subset\rn$ and $\ch_{\fz}^{\dz}$--a. e. $x\in 2Q$,
\[
\sum_{j\in\zz}\ch_{\fz}^{\dz}(  W_j\cap Q)w(x)
\le [w]_{\rh_{\fz,\dz}} w_{\ch_{\fz}^{\dz}}(Q).\]
Taking the Choquet integral over $E$ with respect to $\ch_{\fz}^{\dz}$ yields \eqref{eq0902-1} for $p=\fz$.

If $p\in(1,\fz)$, then by H\"older's inequality and the definition of $\rh_{p,\dz}$,
\[
\begin{aligned}
\int_{E}w(x)\,d\ch_{\fz}^{\dz}&\le 2 \left\{\int_{E}w(x)^{p}\,d\ch_{\fz}^{\dz}\right\}^{\frac1p}\ch_{\fz}^{\dz}(E)^{1-\frac1p}\le2 \left\{\int_{2Q}w(x)^{p}\,d\ch_{\fz}^{\dz}\right\}^{\frac1p}\ch_{\fz}^{\dz}(E)^{1-\frac1p}\\
&\le2 [w]_{\rh_{p,\dz}}w_{\ch_{\fz}^{\dz}}(Q)\left\{\frac{\ch_{\fz}^{\dz}(E)}{\sum_{j\in\zz}\ch_{\fz}^{\dz}(  W_j\cap Q)}\right\}^{1-\frac1p},
\end{aligned}
\]
which implies \eqref{eq0902-1} for $p\in(1,\fz)$, and hence (ii) holds.
\end{proof}

\begin{rem}\label{09re13}
	Let $q\in(1,\fz)$. Then 
	$w\in \rh_{q,\dz}$ if and only if $w$ satisfies the reverse H\"older inequality \eqref{8e26}.
	To achieve this, by Remark \ref{re109}, it suffices to show that \eqref{8e26} implies $w\in \rh_{q,\dz}$.
	Indeed, using Proposition \ref{prop-0902-1}(ii) and the fact that $\ch_{\fz}^{\dz}(2Q)=2^{\dz}\ch_{\fz}^{\dz}(Q)$ for any $Q\subset\rn$, we deduce that
	\begin{align*}
		&\left\{\sum_{j\in\zz}\ch_{\fz}^{\dz}(    W_j\cap Q)\right\}^{1-\frac{1}{q}}\left\{\int_{2Q}w(x)^q\,d\ch_{\fz}^{\dz}\right\}^{\frac{1}{q}}\\
		&\quad\ls \ch_{\fz}^{\dz}(2Q)^{1-\frac{1}{q}}\left\{\int_{2Q}w(x)^q\,d\ch_{\fz}^{\dz}\right\}^{\frac{1}{q}}
		\ls \ch_{\fz}^{\dz}(2Q)\left\{\fint_{2Q}w(x)\,d\ch_{\fz}^{\dz}\right\}
		\sim  w_{\ch_{\infty}^{\delta}}(Q)
	\end{align*}
	which implies $w\in \rh_{q,\dz}$.
	Similarly, $w\in\rh_{\infty,\dz}$ if and only if $w$ satisfies \eqref{8e30}.
\end{rem}

From \eqref{eq1113-5} and \eqref{eq0902-1}, we immediately deduce the following lemma.

\begin{lem}\label{lm0922-2}
Let $p\in(1,\fz]$ and  $w\in\rh_{p,\dz}$. Then for any cube $Q$ of $\rn$ and any subset $E\subset Q$,
\begin{align}\label{eq0922-4}
\frac{w_{\ch_{\fz}^{\dz}}(E)}{w_{\ch_{\fz}^{\dz}}(Q)}\leq 2[w]_{\rh_{p,\dz}}\left[\frac{\ch_{\fz}^{\dz}(E)}{\ch_{\fz}^{\dz}(Q)}\right]^{1-\frac1p}.
\end{align}
\end{lem}

\subsection{Capacitary Weighted Function Spaces}\label{s11.11-6}

For any $\ch_\fz^\delta$-capacitary weight $w$ and $E\subset\rn$,
we define the \emph{weighted Choquet--Lebesgue space}
$L^p_w(E,\ch_\fz^\delta)$ with $p\in[1,\fz)$ as the set of all functions $f$ satisfying
\begin{align}\label{eq319-1}
\|f\|_{L^p_w(E,\ch_\fz^\delta)}
:=&\lf\{\int_E |f(x)|^pw(x)\,d\ch_\fz^\delta \r\}^{\frac{1}{p}}<\fz,
\end{align}
and the space $L^p(E,w_{\mathcal H^{\delta}_\infty})$ with $p\in[1,\fz)$ as the set of all functions $f$ satisfying
\begin{align}\label{eq319-1'}
\|f\|_{L^p(E,w_{\mathcal H^{\delta}_\infty})}
:=&\lf\{\int_E |f(x)|^p\,dw_{\mathcal H^{\delta}_\infty}\r\}^{\frac{1}{p}}<\fz.
\end{align}
Moreover, define the \emph{weak Choquet--Lebesgue spaces} $L^{p,\fz}(E, w_{\mathcal H^{\delta}_\infty})=L_w^{p,\fz}(E, \ch_{\fz}^{\dz})$ as the set of all functions $f$ satisfying
\begin{align*}
\|f\|:=\sup_{t\in(0,\fz)} t\lf(w_{\mathcal H^{\delta}_\infty}(\{x\in E:\,|f(x)|>t\})\r)^{\frac{1}{p}}<\fz,
\end{align*}
and the \emph{spaces}
$L^{\fz}(E,w_{\mathcal H^{\delta}_\infty})=L^{\fz}_w(E,\ch_\fz^\delta)$ as the collection of all functions $f$ satisfying
\begin{align*}
\|f\|:=\inf(\{t>0:\  w _{\mathcal H^{\delta}_\infty}(\{x\in E:\,|f(x)|>t\})=0\})<\fz,
\end{align*}
where $w _{\mathcal H^{\delta}_\infty}$
is as in \eqref{eq1111-3}.

We now clarify the relationship between
$L^p_w(\rn,\ch_\fz^\delta)$ and  $L^p(\rn,w_{\ch_\fz^\delta})$.

\begin{rem}\label{8rem26}
\begin{enumerate}
\item[{\rm\textbf{(i)}}]\textbf{Strict embedding}: For any $\ch_\fz^\delta$-capacitary weight $w$, the proof of \cite[Proposition 2.7]{HZZ25} infers that $L^{p}(\rn,w_{\ch_\fz^\dz})\subset  L^{p}_w(\rn,\ch_\fz^\dz)$ for $p\in[1,\fz)$ and
\begin{align}\label{eq1112-3}
\|f\|_{L^{p}_w(\rn,\ch_\fz^\dz)}\leq 4^{\frac1p} \|f\|_{L^{p}(\rn,w_{\ch_\fz^\dz})}.
\end{align}
Furthermore, for any general capacity $C$, it also holds
\begin{align}\label{eq1112-3'}
\|f\|_{L^{p}_w(\rn,C)}\ls\|f\|_{L^{p}(\rn,w_{C})}.
\end{align}
However, a counterexample in \cite[Remark 2.9(iii)]{HZZ25} yields that for some weight $w$, there exists function $f\in L^{p}_w(\rn,\ch_\fz^\dz)$ such that $f\notin L^{p}(\rn,w_{\ch_\fz^\dz})$, i.e.,
$$L^{p}(\rn,w_{\ch_\fz^\dz})\subsetneqq  L^{p}_w(\rn,\ch_\fz^\dz).$$

\item[{\rm\textbf{(ii)}}]
\textbf{Characterization of the condition such that $L^{p}_w(\rn,\ch_\fz^\dz)=L^{p}(\rn,w_{\ch_\fz^\dz})$} for $p\in[1,\fz)$:
By establishing Radon--Nikodym-type theorem in \cite[Proposition 2.7 and Remark 2.9(iv)]{HZZ25}, the following statements are equivalent
\begin{enumerate}
\item[{\rm{(ii-1)}}]
For any $f\in L^{p}_w(\rn,\ch_\fz^\dz)$,
$\|f\|_{L^{p}(\rn,w_{\ch_\fz^\dz})}\ls \|f\|_{L^{p}_w(\rn,\ch_\fz^\dz)}$;
\item[{\rm{(ii-2)}}]
For any cube $Q$ of $\rn$,
$\int_{Q}\frac{1}{w(x)}\,dw_{\ch_{\infty}^{\delta}}\ls \ch_{\infty}^{\delta}(Q)$;
\item[{\rm{(ii-3)}}]
For any subset $E$ of $\rn$, 
\begin{equation}\label{eq1113-3}
\sum_{j\in\zz}\ch_{\infty}^{\delta}(E\cap   W_j)\ls\ch_{\infty}^{\delta}(E).	
\end{equation}
\end{enumerate}
Thus, any one of them leads to $L^{p}_w(\rn,\ch_\fz^\dz)=L^{p}(\rn,w_{\ch_\fz^\dz})$ with equivalent quasi-norms.
\end{enumerate}
\end{rem}

\begin{rem}\label{rem1113-1}
For any $\ch_\fz^\delta$-capacitary weight $w$, we point out that the condition \eqref{eq1113-3}
is equivalent to 
\begin{align}\label{eq1113-4}
\sum_{j\in\zz}\ch_{\infty}^{\delta}(Q\cap   W_j)\ls\ch_{\infty}^{\delta}(Q),\quad 
\forall {\rm cube}\ Q\subset\rn.
\end{align}
In fact, the implication \eqref{eq1113-3}$\Rightarrow$ \eqref{eq1113-4} is obvious. 
Conversely, let $E\subset \rn$. If $\ch_{\fz}^{\dz}(E)=0$ or $\ch_{\fz}^{\dz}(E)=\fz$, then \eqref{eq1113-3} holds trivially. Now assume $0<\ch_{\fz}^{\dz}(E)<\fz$. By the definition of the Hausdorff content, there exists countable covering of $E$ by cubes $\{Q_k\}_{k\in\nn}$ such that
\[\sum_{k\in\nn}\ch_{\fz}^{\dz}(Q_k)\leq 2\ch_{\fz}^{\dz}(E).\]
By this and applying condition \eqref{eq1113-4} to each cube $Q_j$, we obtain
\begin{align*}
\sum_{j\in\zz}\ch_{\infty}^{\delta}(E\cap   W_j)&\leq \sum_{j\in\zz}\sum_{k\in\nn}\ch_{\infty}^{\delta}(Q_k\cap   W_j)
=\sum_{k\in\nn}\sum_{j\in\zz}\ch_{\infty}^{\delta}(Q_k\cap   W_j)\\
&\ls\sum_{k\in\nn}\ch_{\infty}^{\delta}(Q_k)\ls \ch_{\infty}^{\delta}(E),
\end{align*}
which establishes \eqref{eq1113-3}. Therefore \eqref{eq1113-3}$\iff$ \eqref{eq1113-4}.

Combining \eqref{eq1113-5}, \eqref{eq0902-2}, and Remark \ref{rem1113-1} yields that for any  $w\in\rh_{q,\dz}$ with $q\in (1,\fz]$ the  equivalence
\[\ch_{\fz}^{\dz}(E)\sim \sum_{j\in\zz}\ch_{\fz}^{\dz}(    W_j\cap E)\]
holds, which is fundamental for establishing norm equivalence in the associated function spaces.
\end{rem}

Combining \eqref{eq0902-2} and Remarks \ref{rem1113-1} and \ref{8rem26}(ii), we immediately conclude the following lemma.

\begin{lem}\label{lm-0903-1}
Let $w \in \rh_{p,\delta}$ with $p\in(1,\fz]$. Then there exists a positive constant $C$ such that,
for any $f\in L_w^1(\rn,\ch_\fz^\delta)$,
\[\int_\rn |f(x)|\,dw_{\mathcal H^{\delta}_\infty}
\leq C\int_{\mathbb R^n}|f(x)|w(x)\,d\mathcal H^{\delta}_\infty.\]
\end{lem}

\begin{rem}\label{8rem27}
Let $w \in \rh_{p,\delta}$ with $p\in(1,\fz)$. As a consequence of Lemma  \ref{lm-0903-1}, we have
\[\int_{Q}\frac{1}{w(x)}\,dw_{\ch_{\infty}^{\delta}}\ls \ch_{\infty}^{\delta}(Q)\]
and
\[\int_{Q}w(x)^{p-1}\,dw_{\ch_{\infty}^{\delta}}\ls
\int_{Q}w(x)^p\,d\mathcal H^{\delta}_\infty
\]
for any cube $Q$ of $\rn$. These two inequalities will be frequently used.
\end{rem}

Furthermore, the spaces
$L^p_w(\rn,\ch_\fz^\delta)$ and  $L^p(\rn,w_{\ch_\fz^\delta})$ are indeed equivalent for $w \in \rh_{p,\delta}$.

\begin{rem}\label{8rem29}
Propositions \ref{prop-0901-1}(ii)
and \ref{prop-0902-1}(ii)
respectively show that
	for any $w\in\ca_{p,\dz}$ or $w\in\rh_{p,\dz}$ with $p\in (1,\fz]$, \eqref{eq1113-4} holds. Thus, from Remarks \ref{rem1113-1} and \ref{8rem26}(ii), we infer that, for $w\in\{\ca_{p,\dz},\rh_{p,\dz}\}$ with $p\in (1,\fz]$, $L^{p}_w(\rn,\ch_\fz^\dz)=L^{p}(\rn,w_{\ch_\fz^\dz})$ and there exists $W_0>0$ such that
\[{\frac14}\int_{\mathbb R^n}|f(x)|^pw(x)\,d\mathcal H^{\delta}_\infty\le \int_\rn |f(x)|^p\,dw_{\mathcal H^{\delta}_\infty}
\le W_0 \int_{\mathbb R^n}|f(x)|^pw(x)\,d\mathcal H^{\delta}_\infty.\]
\end{rem}

\subsection{Relation Between $\rh_{q,\delta}$ and $\ca_{p,\dz}$}\label{s11.11-9}

A well-known result of Coifman and Fefferman \cite{CoFe74} established that
$w\in A_{\fz}$ if and only if $w\in RH_{1}$, which forms a fundamental and crucial connection between Muckenhoupt $A_{p}$-weights and reverse H\"older classes.

In this subsection, we
establish the equivalence between the capacitary reverse H\"older class $\rh_{1,\dz}$ and the capacitary Muckenhoupt class $\ca_{\fz,\dz}$  through a different method, which goes beyond Lebesgue measures. This result plays an important role in the sequel.

\begin{thm}\label{prop-0916-1}
Let $\delta\in(0,n]$. Then $w\in\ca_{\fz,\dz}$ if and only if $w\in\rh_{1,\dz}$.
\end{thm}

To prove Theorem \ref{prop-0916-1}, we need several lemmas.
The following one is an extension of \cite[Lemma 4.5]{HZZ25}, and we omit the details here.
Recall that $\ca_{p, w_{\ch_{\fz}^{\dz}}}$ is defined as in Section \ref{s11.11-2}.

\begin{lem}\label{lm-0904-1}
Let $w$ be $\ch_\fz^\delta$-capacitary weight and $F\in\ca_{p, w_{\ch_{\fz}^{\dz}}}$ with $p\in(1,\fz)$. Then, for any $\beta\in(0,[F]_{\ca_{p, w_{\ch_{\fz}^{\dz}}}}^{-1})$ and any cube $Q\subset\rn$,
\[w_{\ch_{\infty}^{\delta}}\lf(\lf\{x\in Q:\ F(x)>\beta  F_{Q,w}\r\}\r)\geq \lf[1-\lf([F]_{\ca_{p, w_{\ch_{\fz}^{\dz}}}}\beta\r)^{\frac{1}{p-1}}\r]w_{\ch_{\infty}^{\delta}}(Q),\]
where $ F_{Q,w}:=\fint_{Q}F(x)\,dw_{\ch_{\infty}^{\delta}}$.
\end{lem}

Applying an argument similar to the proof of \cite[Lemma 3.8]{ZZH-2},
we obtain the following result.

\begin{lem}\label{lm-0904-5}
Let $Q$ be a cube in $\rn$. Suppose that $w$ is a $\ch_\fz^\delta$-capacitary weight satisfying  \eqref{eq0902-3} and \eqref{eq0902-2}. Let $F\in L^1_{w}(Q,{\ch}_\fz^{\delta})$. Then, for any $\lambda\ge\fint_{Q}|F(x)|\,dw_{{\ch}_\fz^{\delta}}$, there exist a collection of pairwise disjoint dyadic subcubes $\{Q_j\}_{j\in\nn}$ of $Q$ and positive constants $C_1$ and $C_2$ such that the following hold:
\begin{enumerate}
\item[{\rm(i)}]
for any $j\in\nn$,
$\lz<\fint_{Q_j}|F(x)|\,dw_{{\ch}_\fz^{\delta}}\le C_1\lambda;$
\item[{\rm(ii)}]
for ${\ch}_\fz^{\delta}$-almost every $x\in Q\backslash\cup_{j\in\nn}Q_j$,
$|F(x)|\leq C_2\lambda.$
\end{enumerate}
\end{lem}

The following lemma comes from \cite[Lemma 2, (6) and (7)]{0923-2}.

\begin{lem}\label{lm-0904-4}
Let $Q$ be a cube in $\rn$. If $\{Q_j\}_{j\in\nn}$ is a family of non-overlapping dyadic subcubes of $Q$, then there exists a subfamily $\{Q_{j_v}\}_v$ such that
\begin{enumerate}
\item[{\rm(i)}]
$\ch_{\fz}^{\dz}(\cup_{j\in\nn}Q_j)\leq 2 \sum_{v}\ch_{\fz}^{\dz}(Q_{j_v})$;
\item[{\rm(ii)}]
there exists a positive constant $C$ such that for any $f\in L^1({\cup_{v\in\nn}Q_{j_v}, \ch_{\fz}^{\dz}})$,
\[\sum_{v}\int_{Q_{j_v}}|f(x)|\,d\ch_{\fz}^{\dz}\leq C\int_{\cup_{v}Q_{j_v}}|f(x)|\,d\ch_{\fz}^{\dz}.\]
\end{enumerate}
\end{lem}

We now show Theorem \ref{prop-0916-1}.

\begin{proof}[Proof of Theorem \ref{prop-0916-1}]
We begin with showing that if $w\in\ca_{\fz,\dz}$, then $w\in\rh_{1,\dz}$.
Since $w\in\ca_{\fz,\dz}$, there exists $p\in[1,\fz)$ such that $w\in\ca_{p,\dz}$. By (iv) and  (vi) of Proposition \ref{prop-0901-1}, there exist positive constants $\gamma_{1}$ and $\gamma$ with $\gamma\leq \gamma_{1}$ such that $w^{\gamma+1}\in\ca_{p,\dz}$ and for any cube $Q\subset\rn$,
\begin{align}\label{eq-0903-1}
\lf[\fint_{Q}w(x)^{1+\gamma_1}\,d\ch_{\infty}^{\delta}\r]^{\frac{1}{1+\gamma_1}}\ls \fint_{Q}w(x)\,d\ch_{\infty}^{\delta}.
\end{align}
By H\"older's inequality, the last inequality remains valid if we replace $\gamma_1$ by $\gamma$. Using this, $w^{\gamma+1}\in\ca_{p,\dz}$ and Proposition \ref{prop-0901-1}(ii), we obtain
\begin{align*}
&\lf[\sum_{j\in\zz}\ch_{\fz}^{\dz}(  W_j\cap Q)\r]^{1-\frac{1}{1+\gamma}}\lf[\int_{2Q}w(x)^{1+\gamma}\,d\ch_{\fz}^{\dz}\r]^{\frac{1}{1+\gamma}}\ls \ch_{\fz}^{\dz}(Q)\lf[\fint_{Q}w(x)^{1+\gamma}\,d\ch_{\fz}^{\dz}\r]^{\frac{1}{1+\gamma}} \ls w_{\ch_{\infty}^{\delta}}(Q),
\end{align*}
which implies that $w\in\rh_{1+\gamma, \dz}\subset \rh_{1, \dz}$.

Conversely, we prove that if $w\in\rh_{1,\dz}$, then $w\in\ca_{\fz,\dz}$.
Since $w\in\rh_{1,\dz}$, there exists $q\in(1,\fz]$ such that $w\in\rh_{q,\dz}$. Notice that $\rh_{\fz,\dz}\subset \rh_{q,\dz}$ for any $q\in(1,\fz)$ due to Proposition \ref{prop-0902-1}(i).
We may assume $w\in\rh_{q,\dz}$ for $q\in(1,\fz)$ in what follows.
Then, for any cube $Q\subset\rn$, by \eqref{eq1113-5},
\begin{align*}
\ch_{\fz}^{\dz}(Q)^{1-\frac{1}{q}}\lf[\int_Q w(x)^q\,d\ch_{\fz}^{\dz}\r]^{\frac{1}{q}}\leq\lf[\sum_{j\in\zz}\ch_{\fz}^{\dz}(  W_j\cap Q)\r]^{1-\frac{1}{q}}\lf[\int_{2Q}w(x)^q\,d\ch_{\fz}^{\dz}\r]^{\frac{1}{q}}\le [w]_{\rh_{q,\dz}} w_{\ch_{\infty}^{\delta}}(Q).
\end{align*}
Taking the $\frac{q}{q-1}$-th power on both sides of the above inequality gives
\[\lf[\frac{1}{w_{\ch_{\fz}^{\dz}}(Q)}\ch_{\fz}^{\dz}(Q)\r]\lf[\frac{1}{w_{\ch_{\fz}^{\dz}}(Q)}\int_{Q}w(x)^q\,d\ch_{\fz}^{\dz}\r]^{\frac{1}{q-1}}\le [w]_{\rh_{q,\dz}}^{\frac{q}{q-1}}.\]
From this and Remark \ref{8rem27}, it follows that, for any cube $Q\subset\rn$,
\begin{align}\label{eq0904-1}
\lf[\frac{1}{w_{\ch_{\fz}^{\dz}}(Q)}\int_Q\frac{1}{w(x)}\,dw_{\ch_{\fz}^{\dz}}\r]\lf[\frac{1}{w_{\ch_{\fz}^{\dz}}(Q)}\int_{Q}w(x)^{q-1}\,dw_{\ch_{\fz}^{\dz}}\r]^{\frac{1}{q-1}}\leq [w]_{\rh_{q,\dz}}^{\frac{q}{q-1}},
\end{align}
which implies that $\frac{1}{w}\in\ca_{q', w_{\ch_{\fz}^{\dz}}}$ with $q':=\frac{q}{q-1}$ and $$\lf[\frac1w\r]_{\ca_{q', w_{\ch_{\fz}^{\dz}}}}\leq [w]_{\rh_{q,\dz}}^{q'}.$$ By this and Lemma \ref{lm-0904-1}, for any $\beta\in(0,[\frac1w]_{\ca_{q', w_{\ch_{\fz}^{\dz}}}}^{-1})$ and any cube $Q\subset\rn$,
\begin{align}\label{eq0904-3}
w_{\ch_{\infty}^{\delta}}\lf(\lf\{x\in Q:\ \frac{1}{w(x)}>\beta \lf(\frac1w\r)_{Q,w}\r\}\r)\geq \lf[1-\lf(\lf[\frac1w\r]_{\ca_{q', w_{\ch_{\fz}^{\dz}}}}\beta\r)^{q-1}\r]w_{\ch_{\infty}^{\delta}}(Q),
\end{align}
where $(\frac1w)_{Q,w}:=\fint_{Q}\frac{1}{w(x)}\,dw_{\ch_{\infty}^{\delta}}$.
Observe that $w$ satisfies \eqref{eq0902-3} and \eqref{eq0902-2} due to Proposition \ref{prop-0902-1}(ii). By Lemma \ref{lm-0904-5}, for any $\lambda\ge (\frac1w)_{Q,w}$, there exist pairwise disjoint dyadic subcubes $\{Q_j\}_{j\in\nn}$ of $Q$ and positive constants $C_1$ and $C_2$ such that
\begin{enumerate}
\item[{\rm(i)}]
for any $j\in\nn$,
$\lz<\fint_{Q_j}\frac{1}{w(x)}\,dw_{{\ch}_\fz^{\delta}}\le C_1\lambda;$
\item[{\rm(ii)}]
for ${\ch}_\fz^{\delta}$-almost every $x\in Q\backslash\cup_{j\in\nn}Q_j$, $\frac{1}{w(x)}\leq C_2\lambda$.
\end{enumerate}
By (ii), Remark \ref{8rem27}, Lemma \ref{lm-0904-4}, \eqref{eq1112-3}, (i), and \eqref{eq0904-3}, there exists a subfamily $\{Q_{j_v}\}_{v}$ of $\{Q_j\}_{j\in\nn}$ such that, for any
$\lambda\ge (\frac1w)_{Q,w}$,
\begin{align}\label{eq-0908-1}
&\lf(w^{-1}\r)_{w_{\ch_{\infty}^{\delta}}}\lf(\lf\{x\in Q: \frac{1}{w(x)}>C_2\lambda\r\}\r)\\
&\quad\leq \int_{\cup_{j\in\nn}Q_j}\frac{1}{w(x)}\,dw_{\ch_{\fz}^{\dz}}\ls \ch_{\fz}^{\dz}(\cup_{j\in\nn}Q_j)\noz\\
&\quad\ls \sum_{v}\ch_{\infty}^{\delta}(Q_{j_v})\ls \sum_{v}\int_{Q_{j_v}}\frac{1}{w(x)}\,dw_{\ch_{\fz}^{\dz}}
\ls \lambda \sum_{v}w_{\ch_{\infty}^{\delta}}(Q_{j_v})\noz\\
&\quad\ls \lambda\sum_{v}w_{\ch_{\infty}^{\delta}}\lf(\lf\{x\in Q_{j_v}:\ \frac{1}{w(x)}>\beta \lf(\frac1w\r)_{Q_{j_v}}\r\}\r)\noz\\
&\quad\ls \lambda\sum_{v}\int_{Q_{j_v}}
w(x)\mathbf{1}_{\{x\in Q:\ w(x)^{-1}>\beta\lambda\}}(x)\,d\ch_{\infty}^{\delta}\noz\\
&\quad\ls \lambda\int_{\cup_v{Q_{j_v}}}
w(x)\mathbf{1}_{\{x\in Q:\ w(x)^{-1}>\beta\lambda\}}(x)\,d\ch_{\infty}^{\delta}\ls\lambda w_{\ch_{\infty}^{\delta}}\lf(\lf\{x\in Q:\ \frac{1}{w(x)}>\beta\lambda\r\}\r),\noz
\end{align}
where $(w^{-1})_{w_{\ch_{\infty}^{\delta}}}$ is defined in \eqref{eq1111-6}.

Next, we claim that for $w\in\rh_{q,\dz}$ with $q\in(1,\fz)$ there exist positive constants $c_1$ and $c_2$ such that, for any cube $Q\subset\rn$ and $\gamma\in(0,1)$,
\begin{equation}\label{eq0908-1}
\int_{Q}w(x)^{-(1+\gamma)}\,dw_{\ch_{\fz}^{\dz}}\leq c_1\gamma \int_Q w(x)^{-(1+\gamma)}\,dw_{\ch^{\delta}_{\infty}}+c_2 \lf(\frac1w\r)_{Q,w}^{1+\gamma}w_{\ch_{\fz}^{\dz}}(Q).
\end{equation}
To show this, we first prove that, for any cube $Q\subset\rn$ and $\gamma\in(0,1)$,
\begin{equation}\label{eq0908-2}
\int_{F_1}\lf\{[C_2w(x)]^{-\gamma}-\lf(\frac1w\r)_{Q,w}^{\gamma}\r\}w(x)^{-1}\,dw_{\ch_{\fz}^{\dz}}\ls\gamma \int_Qw(x)^{-(1+\gamma)}\,dw_{\ch^{\delta}_{\infty}},
\end{equation}
where $F_1:=\{x\in Q: \frac1{w(x)}>C_2\lf(\frac1w\r)_{Q,w}\}$.
Indeed, by \eqref{eq1112-3'} and choosing $\beta=\frac{1}{2[\frac1w]_{\ca_{q', w_{\ch_{\fz}^{\dz}}}}}$ in \eqref{eq-0908-1}, we get
\begin{align*}
&\int_{F_1}\lf\{[C_2w(x)]^{-\gamma}-\lf(\frac1w\r)_{Q,w}^{\gamma}\r\}w(x)^{-1}\,dw_{\ch_{\fz}^{\dz}}\\
&\hs \ls \int_{F_1}[C_2w(x)]^{-\gamma}-\lf(\frac1w\r)_{Q,w}^{\gamma}\,d(w^{-1})_{w_{\ch_{\infty}^{\delta}}}\\
&\hs\sim \int_{0}^{\infty}\lf(w^{-1}\r)_{w_{\ch_{\infty}^{\delta}}}{\lf\{x\in F_1:\ [C_2w(x)]^{-\gamma}-\lf(\frac1w\r)_{Q,w}^{\gamma}>t\r\}}\,dt\\
&\hs \sim\gamma\int_{\lf(\frac1w\r)_{Q,w}}^{\infty}t^{\gamma-1}\lf(w^{-1}\r)_{w_{\ch_{\infty}^{\delta}}}{\lf\{x\in F_1:\ [C_2w(x)]^{-\gamma}-\lf(\frac1w\r)_{Q,w}^{\gamma}>t^{\gamma}-\lf(\frac1w\r)_{Q,w}^{\gamma}\r\}}\,dt\\
&\hs \sim\gamma\int_{\lf(\frac1w\r)_{Q,w}}^{\infty}t^{\gamma-1}\lf(w^{-1}\r)_{w_{\ch_{\infty}^{\delta}}}\{x\in Q: w(x)^{-1}>C_2t\}\,dt\\
&\hs \ls\gamma\int_{\lf(\frac1w\r)_{Q,w}}^{\infty}t^{\gamma}w_{\ch_{\infty}^{\delta}}\lf\{x\in Q: \frac{1}{w(x)}>\frac{t}{2[\frac1w]_{\ca_{q', w_{\ch_{\fz}^{\dz}}}}}\r\}\,dt \ls \gamma\int_{Q}w(x)^{-(1+\gamma)}\,dw_{\ch^{\delta}_{\infty}},
\end{align*}
which yields \eqref{eq0908-2}.

In addition, from (i) and (iii) of Remark \ref{rem-11112-1} and the definition of $F_1$, we infer that, for any $Q\subset \rn$,
\begin{align}\label{eq-0911-1}
&\int_{Q}w(x)^{-(1+\gamma)}\,dw_{\ch^{\delta}_{\infty}}\leq \int_{Q\backslash F_1}w(x)^{-(1+\gamma)}\,dw_{\ch^{\delta}_{\infty}}+\int_{F_1}w(x)^{-(1+\gamma)}\,dw_{\ch^{\delta}_{\infty}}\\
&\hs \ls \lf(\frac1w\r)_{Q,w}^{1+\gamma}w_{\ch_{\infty}^{\delta}}(Q)+\int_{F_1}\lf\{[C_2w(x)]^{-\gamma}-\lf(\frac1w\r)_{Q,w}^{\gamma}\r\}\frac1{w(x)}\,dw_{\ch_{\infty}^{\delta}}
+\int_{F_1}\lf(\frac1w\r)_{Q,w}^{\gamma}\frac1{w(x)}\,dw_{\ch_{\infty}^{\delta}}\noz\\
&\hs \ls\int_{F_1}\lf\{[C_2w(x)]^{-\gamma}-\lf(\frac1w\r)_{Q,w}^{\gamma}\r\}\frac1{w(x)}\,dw_{\ch_{\infty}^{\delta}}+ \lf(\frac1w\r)_{Q,w}^{1+\gamma}w_{\ch_{\infty}^{\delta}}(Q),\noz
\end{align}
which, combined with \eqref{eq0908-2} and \eqref{eq-0911-1}, further implies \eqref{eq0908-1}. Thus, the claim is proved.

For any $N\in(0,\fz)$, define
\begin{align}\label{9eq5}
w_N(x):=
\begin{cases}
w(x),\quad &x\in G_1:=\{x\in\rn: w(x)\ge N\},\\
N,\quad &x\in G_2:=\{x\in\rn: w(x)< N\}.
\end{cases}
\end{align}
For any $j\in\zz$, set
\[W_{N,j}:=\lf\{x\in\rn: 2^{j-1}<w_N(x)\leq 2^j\r\}.\]
Let $N_0\in \zz$ satisfy $2^{N_0-1}<N\leq 2^{N_0}$. Then $W_{N,j}=\emptyset$ if $j\leq N_0-1$, $W_{N,j}=G_2\cup(G_1\cap w_{N_0})$ if $j=N_0$, and $W_{N,j}=  W_j$ if $j\ge N_0+1$. Moreover,
\[G_2\cup(G_1\cap w_{N_0})=\lf(\bigcup_{j\in (-\fz, N_0]\cap \zz}   W_j\r)\cup\lf\{x\in\rn: w(x)=0\r\}.\]
Thus, for any cube $Q\subset\rn$,
\begin{align}\label{eq831-1}
\sum_{j\in\zz}\ch_{\fz}^{\dz}(W_{N,j}\cap Q)\leq\sum_{j\in\zz}\ch_{\fz}^{\dz}(w_{j}\cap Q).
\end{align}
From this, Remark \ref{rem-11112-1}(iv), and \eqref{eq0902-2}, it follows that, for any cube $Q\subset\rn$,
\begin{align*}
&\lf[\sum_{j\in\zz}\ch_{\fz}^{\dz}(W_{N,j}\cap Q)\r]^{1-\frac{1}{q}}\lf[\int_{2Q}(w_N)(x)^{q}\,d\ch_{\fz}^{\dz}\r]^{\frac{1}{q}}\\
&\quad\leq\lf[\sum_{j\in\zz}\ch_{\fz}^{\dz}(w_{j}\cap Q)\r]^{1-\frac{1}{q}}\lf[\int_{2Q}\bigl(w(x)^{q}+N^q\bigr)\,d\ch_{\fz}^{\dz}\r]^{\frac{1}{q}}\\
&\quad\leq \lf[\sum_{j\in\zz}\ch_{\fz}^{\dz}(w_{j}\cap Q)\r]^{1-\frac{1}{q}}\lf[\int_{2Q }w(x)^{q}\,d\ch_{\fz}^{\dz}\r]^{\frac{1}{q}}+\lf[\sum_{j\in\zz}\ch_{\fz}^{\dz}(w_{j}\cap Q)\r]^{1-\frac{1}{q}}\lf[\int_{2Q}N^q\,d\ch_{\fz}^{\dz}\r]^{\frac{1}{q}}\\
&\quad \ls\int_{Q}w(x)\,d\ch_{\fz}^{\dz}+\int_{Q}N\,d\ch_{\fz}^{\dz}\ls\int_{Q}w_{N}(x)\,d\ch_{\fz}^{\dz},
\end{align*}
which implies that $w_N\in\rh_{q,\dz}$ with a constant independent of $N$.
Thus, for this $w_N$, there exist positive constants $c_1$ and $c_2$ such that \eqref{eq0908-1} holds.
By Remark \ref{8rem27},
\[\int_Q w_N(x)^{-(1+\gamma)}\,d(w_N)_{\ch^{\delta}_{\infty}}\ls \int_Q w_N(x)^{-\gamma}\,d\ch^{\delta}_{\infty}\ls N^{-\gamma}\ch^{\delta}_{\infty}(Q)< \fz.\]
Applying \eqref{eq1112-3}, rearranging \eqref{eq0908-1} and choosing $\gamma\in (0,1)$ such that $c_1\gamma<1$, and using Remark \ref{8rem27}, we obtain
\[\int_{Q}w_N(x)^{-\gamma}\,d\ch_{\fz}^{\dz}
\ls\int_Q w_N(x)^{-(1+\gamma)}\,d(w_N)_{\ch^{\delta}_{\infty}}\ls \frac{\ch_{\fz}^{\dz}(Q)^{1+\gamma}}{(w_N)_{\ch_{\fz}^{\dz}}(Q)^{\gamma}}\ls \frac{\ch_{\fz}^{\dz}(Q)^{1+\gamma}}{w_{\ch_{\fz}^{\dz}}(Q)^{\gamma}}.\]
Taking $N:=2^{-m}$ with $m\in\zz$ in \eqref{9eq5} and $w_m(x):=w_{2^{-m}}(x)$, then $0\leq (w_m)^{-\gamma}\uparrow w^{-\gamma}$ as $m\to\fz$. By (ii) and (iv) of Remark \ref{12-4} and letting $m\to\fz$, we deduce
\[\int_{Q}w(x)^{-\gamma}\,d\ch_{\fz}^{\dz}\le \int_{Q}w(x)^{-\gamma}\,d\widetilde\ch_{\fz}^{\dz}= \lim_{m\to\fz}\int_{Q}w_m(x)^{-\gamma}\,d\widetilde\ch_{\fz}^{\dz}\ls \frac{\ch_{\fz}^{\dz}(Q)^{1+\gamma}}{w_{\ch_{\fz}^{\dz}}(Q)^{\gamma}}.\]
This is equivalent to
\[\lf\{\fint_Qw(x)\,d\mathcal H^{\delta}_\infty\r\}\lf\{\fint_Qw(x)^{-\gamma}\,d\mathcal H^{\delta}_\infty\r\}^{\frac{1}{\gamma}}\ls 1,\]
which implies that $w\in\ca_{\frac{\gamma+1}{\gamma}, \dz}\subset\ca_{\fz,\dz}$.
This completes the proof of Theorem \ref{prop-0916-1}.
\end{proof}

Besides Theorem \ref{prop-0916-1}, the following lemma also gives a key relation on these two concepts.

\begin{lem}\label{lm0922-1}
Let $s\in[1,\fz)$ and $q\in [1,\fz]$. Then $w\in\ca_{q, \dz}\cap\rh_{s, \dz}$ if and only if $w^s\in\ca_{s(q-1)+1, \dz}$.
\end{lem}

\begin{proof}
We prove the lemma by considering four cases that cover all possible values of $s$ and $q$.

\textbf{Case 1)} $s=1$ and $q\in [1, \fz]$. The lemma holds trivially by Theorem \ref{prop-0916-1}.

\textbf{Case 2)} $s\in(1,\fz)$ and $q=1$. If $w\in\ca_{1, \dz}\cap\rh_{s, \dz}$, then by \eqref{8e26} for any cube $Q$ of $\rn$, we have
\[\fint_Qw(y)^s\,d\mathcal H^{\delta}_\infty\ls \lf[{\fint_Qw(y)\,d\ch_{\fz}^{\dz}}\r]^{s}\ls w(x)^s,\ \text{for}\ \ch_\fz^\delta{\rm-a.\, e.} \ x\in Q.\]
This implies $w^s\in\ca_{1,\dz}$.

Conversely, if $w^s\in\ca_{1,\dz}$, then applying H\"{o}lder's inequality, we obtain,
for any cube $Q$ of $\rn$, 
\[\fint_Qw(y)\,d\mathcal H^{\delta}_\infty\ls \lf[{\fint_Qw(y)^s\,d\ch_{\fz}^{\dz}}\r]^{\frac1s}\ls w(x),\ \text{for}\ \ch_\fz^\delta{\rm-a.\, e.} \ x\in Q,\]
which yields $w\in\ca_{1,\dz}$. Moreover, by \eqref{eq1111-1}, \eqref{eq1111-2},  and $w^s\in\ca_{1,\dz}$, for $\ch_\fz^\delta{\rm-a.\, e.} \ x\in Q$, we get
\[\lf[\sum_{j\in\zz}\ch_{\fz}^{\dz}(  W_j\cap Q)\r]^{1-\frac{1}{s}}\lf[\int_{2Q}w(x)^s\,d\ch_{\fz}^{\dz}\r]^{\frac{1}{s}}\ls\ch_{\fz}^{\dz}(Q) \lf[{\fint_Qw(x)^s\,d\ch_{\fz}^{\dz}}\r]^{\frac1s}\ls\ch_{\fz}^{\dz}(Q) w(x).\]
Taking the Choquet integral of both sides over $Q$, we immediately conclude $w\in\rh_{s,\dz}$.

\textbf{Case 3)} $s\in(1,\fz)$ and $q\in(1,\fz)$. If $w\in\ca_{q, \dz}\cap\rh_{s, \dz}$, then by \eqref{8e26}, for any cube $Q$ of $\rn$, we have
\begin{align*}
&\lf[\fint_Qw(x)^s\,d\mathcal H^{\delta}_\infty\r]\lf[\fint_Qw(x)^{-\frac{1}{q-1}}\,d\mathcal H^{\delta}_\infty\r]^{s(q-1)}\ls \lf[\fint_Qw(x)\,d\mathcal H^{\delta}_\infty\r]^{s}\lf[\fint_Qw(x)^{-\frac{1}{q-1}}\,d\mathcal H^{\delta}_\infty\r]^{s(q-1)}\ls 1,
\end{align*}
which implies $w^s\in\ca_{s(q-1)+1, \dz}$.

On the other hand, if $w^s\in\ca_{s(q-1)+1, \dz}$,
then applying H\"{o}lder's inequality, we find that, for any cube $Q$ of $\rn$,
\begin{align*}
&\lf[\fint_Qw(x)\,d\mathcal H^{\delta}_\infty\r]\lf[\fint_Qw(x)^{-\frac{1}{q-1}}\,d\mathcal H^{\delta}_\infty\r]^{(q-1)}\leq 2\lf[\fint_Qw(x)^s\,d\mathcal H^{\delta}_\infty\r]^{\frac{1}{s}}\lf[\fint_Qw(x)^{-\frac{1}{q-1}}\,d\mathcal H^{\delta}_\infty\r]^{(q-1)}\ls 1,
\end{align*}
which yields $w\in\ca_{q, \dz}$.

In addition, applying H\"{o}lder's inequality, we get
\[\ch_{\fz}^{\dz}(Q)^q=\lf[\int_{Q}w(x)^{\frac1q}w(x)^{-\frac1q}\,d\ch_{\fz}^{\dz}\r]^q\leq 2^q\lf[\int_Qw(x)\,d\mathcal H^{\delta}_\infty\r]\lf[\int_Qw(x)^{-\frac{1}{q-1}}\,d\mathcal H^{\delta}_\infty\r]^{q-1}.\]
Combining this with \eqref{eq1111-2}, \eqref{eq1111-1}, and $w^s\in\ca_{s(q-1)+1, \dz}$, we conclude that, for any cube $Q$ of $\rn$,
\begin{align*}
&\lf[\sum_{j\in\zz}\ch_{\fz}^{\dz}(  W_j\cap Q)\r]^{1-\frac{1}{s}}\lf[\int_{2Q}w(x)^s\,d\ch_{\fz}^{\dz}\r]^{\frac{1}{s}}\ls\ch_{\fz}^{\dz}(Q)^{1-\frac1s} \lf[\int_Qw(x)^s\,d\ch_{\fz}^{\dz}\r]^{\frac1s}\\
&\quad\ls\ch_{\fz}^{\dz}(Q)^q\lf[\fint_Qw(x)^s\,d\ch_{\fz}^{\dz}\r]^{\frac1s}\lf[\fint_Qw(x)^{-\frac{1}{q-1}}\,d\ch_{\fz}^{\dz}\r]^{q-1}\lf[\int_Qw(x)^{-\frac{1}{q-1}}\,d\ch_{\fz}^{\dz}\r]^{1-q}\\
&\quad \ls \ch_{\fz}^{\dz}(Q)^q\lf[\int_Qw(x)^{-\frac{1}{q-1}}\,d\ch_{\fz}^{\dz}\r]^{1-q}\ls \int_Qw(x)\,d\mathcal H^{\delta}_\infty.
\end{align*}
Thus, $w\in\rh_{s,\dz}$.

\textbf{Case 4)} $s\in(1,\fz)$ and $q=\fz$. If $w\in\ca_{\fz, \dz}\cap\rh_{s, \dz}$, then by the definition of $\ca_{\fz, \dz}$, there exists $p\in[1, \fz)$ such that $w\in\ca_{p, \dz}\cap\rh_{s, \dz}$. From Case 2 if $p=1$ or Case 3 if $p\in(1,\fz)$, we conclude $w^s\in\ca_{s(p-1)+1, \dz}\subset\ca_{\fz,\dz}$.

Conversely, if $w^s\in\ca_{\fz, \dz}$, then there exists $p\in[1, \fz)$ such that $w^s\in\ca_{p, \dz}$. Let $\widetilde{p}:=\frac{p-1}{s}+1$, so $w^s\in\ca_{s(\widetilde{p}-1)+1, \dz}$. Applying Case 2 if $\widetilde{p}=1$ or Case 3 if $\widetilde{p}\in(1,\fz)$, we get $$w\in\ca_{\widetilde{p},\dz}\cap\rh_{s,\dz}\subset\ca_{\fz,\dz}\cap\rh_{s,\dz}.$$
Combining the above four cases, Lemma \ref{lm0922-1} holds true.
\end{proof}

\subsection{Self-Improving Properties of $\rh_{q,\delta}$}\label{s11.11-8}

From Proposition \ref{prop-0902-1}(i), we dedeuce that $\rh_{q,\dz}\subset\rh_{p,\dz}$ for any $p<q$. More essentially, in this subsection, we prove that the exponent $q$ can also be improved, meaning that the weight class $\rh_{q,\delta}$ has the self-improving property or the so-called openness property as follows.

\begin{thm}\label{lm-0921-1}
Let $w\in \rh_{q,\delta}$ with $q\in(1,\infty)$. Then there exists $\varepsilon>0$
such that $w\in \rh_{q+\varepsilon,\delta}$.
\end{thm}

We first show the following reverse H\"older inequality associated with $w_{\ch_{\fz}^{\dz}}$, which is a key observation to establish Theorem \ref{lm-0921-1}.

\begin{lem}\label{lm-0919-1}
Let $w\in\ca_{p,\dz}$ with $p\in[1,\fz)$ and $g\in\ca_{q, w_{\ch_{\fz}^{\dz}}}$ with $q\in(1,\fz)$. Then there exist positive constants $K$ and $\gamma\in(0,1)$ such that, for every cube $Q$ of $\rn$,
\begin{align*}
\lf[\fint_{Q}g(x)^{1+\gamma}\,dw_{\ch_{\infty}^{\delta}}\r]^{\frac{1}{1+\gamma}}\leq K\fint_{Q}g(x)\,dw_{\ch_{\infty}^{\delta}}.
\end{align*}
\end{lem}

To prove this reverse H\"older inequality, we need some lemmas. Firstly, by \cite[Lemma 3.5]{ZZH-2} and the proof of \cite[Lemma 2.11]{HZZ25}, we have the following lemma.

\begin{lem}\label{lm-0917-3}
Let $w\in \ca_{p,\delta}$ with $p\in [1,\fz)$, $g\in\ca_{q, w_{\ch_{\fz}^{\dz}}}$ with $q\in(1,\fz)$, and $E\subset \rn$   satisfying $\int_{E}g(x)\,dw_{\ch_{\fz}^{\dz}}<\infty$. Then there exist a family $\{Q_j\}_{j\in\nn}$ of non-overlapping dyadic cubes and a subset $F\subset \rn$ with $w_{\ch_\fz^\delta}(F)=0$ such that
$$E\subset\lf(\bigcup_{j\in\nn} Q_j\r) \bigcup F\quad {\rm and} \quad \sum_{j\in\nn} \int_{Q_j}g(x)\,dw_{\ch_{\fz}^{\delta}}\leq K\int_E g(x)\,dw_{\ch_{\fz}^{\delta}},$$
where $K$ is a positive constant independent of $E$.
\end{lem}

Combining  Lemma \ref{lm-0917-3} with the proof of \cite[Proposition 2.12]{HZZ25}, we conclude the following lemma.

\begin{lem}\label{lm-0917-1}
Let $w\in\ca_{p,\dz}$ with $p\in[1,\fz)$ and $g\in\ca_{q, w_{\ch_{\fz}^{\dz}}}$ with  $q\in(1,\fz)$. Let $\{Q_j\}_{j\in\nn}$ be a family of non-overlapping dyadic cubes of $\rn$. If there exists a constant $\beta\in(0,\fz)$ such that, for each dyadic cube $Q$,
\begin{equation}\label{eq11-11-b}
\sum_{Q_{j}\subset Q}g_{w_{\mathcal H^\delta_\infty}}(Q_{j})\le \beta\,g_{w_{\mathcal H^\delta_\infty}}(Q),
\end{equation}
then there exists a positive constant $K$ such that
\begin{equation}\label{eq11-11-a}
\sum_{j\in\nn}\int_{Q_{j}}|f(x)|\,dg_{w_{\mathcal H^{\delta}_{\infty}}}\le K\int_{\cup_{j\in\nn} Q_{j}}|f(x)|\,dg_{w_{\mathcal H^{\delta}_{\infty}}},
\end{equation}
here and in the sequel, for any $E\subset\rn$,
$g_{w_{\ch_{\fz}^{\dz}}}(E):=\int_{E}g(x)\,dw_{\ch_{\fz}^{\dz}}.$
\end{lem}

In addition, we have the following lemma via an argument similar to that used in the proof of \cite[Lemma 2.6]{HZZ25}.

\begin{lem}\label{lm-0917-4}
Let $w$ and $g$ be $\ch_\fz^\delta$-capacitary weights. Assume that $\{Q_j\}_{j\in\nn}$ is a family of non-overlapping dyadic cubes of $\rn$, then there exists a subfamily $\{Q_{j_v}\}_{v=1}^N$, with positive integer $N<\fz$ or $N=\fz$, satisfying
\begin{enumerate}
\item[{\rm(i)}]
for any dyadic cubes $Q$,
$\sum_{Q_{j_v}\subset Q} g_{w_{\ch_\fz^\delta}}(Q_{j_v})\le 2 g_{w_{\ch_\fz^\delta}}(Q);$

\item[{\rm(ii)}]
$g_{w _{\mathcal H^\delta_\infty}} (\bigcup_{j\in\nn}Q_j)
\le 2\sum_{v=1}^N g_{w _{\mathcal H^\delta_\infty}}(Q_{j_v}).$
\end{enumerate}
\end{lem}

We next show Lemma \ref{lm-0919-1}.

\begin{proof}[Proof of Lemma \ref{lm-0919-1}]
By the proof of \eqref{eq0908-1} with $\frac{1}{w}$ replaced by $g$ and Lemma \ref{lm-0904-4} therein replaced by Lemmas \ref{lm-0917-1} and \ref{lm-0917-4}, we then deduce that there exist positive constants $k_1$ and $k_2$ such that, for any cube $Q\subset\rn$ and any $\gamma\in(0,1)$,
\begin{align}\label{eq0919-2}
\int_{Q}g(x)^{1+\gamma}\,dw_{\ch_{\fz}^{\dz}}\leq k_1\gamma \int_Q g(x)^{1+\gamma}\,dw_{\ch^{\delta}_{\infty}}+ k_2w_{\ch_{\fz}^{\dz}}(Q)\lf[\fint_{Q}g(x)\,dw_{\ch_{\fz}^{\dz}}\r]^{1+\gamma}.
\end{align}

For any $m\in\nn$, define
$$g_m(x):=
\begin{cases}
m,\quad &x\in G_1:=\{x\in\rn: g(x)>m\},\\
g(x),\quad &x\in G_2:=\{x\in\rn: g(x)\leq m\}.
\end{cases}
$$
Then $g_m\le g$, and we next prove that $g_m\in\ca_{q,w_{\ch_{\fz}^{\dz}}}$ with a constant independent of $m$.
Indeed, for any cube $Q\subset\rn$ and any $m\in\nn$,
\begin{align*}
&\lf\{\fint_Qg_m(x)\,dw_{\mathcal H^{\delta}_\infty}\r\}\lf\{\fint_Qg_m(x)^{-\frac{1}{q-1}}\,dw_{\mathcal H^{\delta}_\infty}\r\}^{q-1}\\
&\quad\leq \max\lf\{2^{q-1},1\r\}\lf[\lf\{\fint_Qg(x)\,dw_{\mathcal H^{\delta}_\infty}\r\}\lf\{\fint_Qg(x)^{-\frac{1}{q-1}}\,dw_{\mathcal H^{\delta}_\infty}\r\}^{q-1}+1\r]\\
&\quad \leq \max\lf\{2^{q-1},1\r\}\lf([g]_{\ca_{q,w_{\ch_{\fz}^{\dz}}}}+1\r).
\end{align*}
Thus \eqref{eq0919-2} holds also for $g_m$.
Note that for any cube $Q\subset\rn$ and any $m\in\nn$,
\[
\int_Q g_m(x)^{1+\gamma}\,dw_{\ch^{\delta}_{\infty}}\leq m^{1+\gamma}w_{\ch^{\delta}_{\infty}}(Q)< \fz.
\]
Therefore, rearranging \eqref{eq0919-2} for $g_m$ and choosing $\gamma\in (0,1)$ such that $k_1\gamma<1$, we obtain
\begin{align}\label{eq0919-6}
\int_{Q}g_m(x)^{1+\gamma}\,dw_{\ch_{\fz}^{\dz}}\ls w_{\ch_{\fz}^{\dz}}(Q)\lf[\fint_{Q}g_m(x)\,dw_{\ch_{\fz}^{\dz}}\r]^{1+\gamma}\ls w_{\ch_{\fz}^{\dz}}(Q)\lf[\fint_{Q}g(x)\,dw_{\ch_{\fz}^{\dz}}\r]^{1+\gamma}.
\end{align}
Notice that $0\leq (g_m)^{1+\gamma}w\uparrow g^{1+\gamma}w$ as $m\to\fz$. Applying Remark \ref{8rem29}, (ii) and (iv) of Remark \ref{12-4}, \eqref{eq0919-6}, and letting $m\to\fz$, we know that, for any cube $Q\subset\rn$,
\begin{align*}
\int_{Q}g(x)^{1+\gamma}\,dw_{\ch_{\fz}^{\dz}}
&\ls\int_{Q}g(x)^{1+\gamma}w(x)\,d\widetilde\ch_{\fz}^{\dz}\sim \lim_{m\to\fz}\int_{Q}g_m(x)^{1+\gamma}w(x)\,d\widetilde\ch_{\fz}^{\dz}\\
&\ls \limsup_{m\to\fz}\int_{Q}g_m(x)^{1+\gamma}w(x)\,d\ch_{\fz}^{\dz}\ls\limsup_{m\to\fz}\int_{Q}g_m(x)^{1+\gamma}\,dw_{\ch_{\fz}^{\dz}}\\
&\ls
w_{\ch_{\fz}^{\dz}}(Q)\lf[\fint_{Q}g(x)\,dw_{\ch_{\fz}^{\dz}}\r]^{1+\gamma}.
\end{align*}
This is the desired reverse H\"older inequality and hence completes the proof of Lemma \ref{lm-0919-1}.
\end{proof}

Now we prove Theorem \ref{lm-0921-1} with the help of Theorem \ref{prop-0916-1} and Lemma \ref{lm-0919-1}.

\begin{proof}[Proof of Theorem \ref{lm-0921-1}]
Since $w\in\rh_{q,\dz}$, it follows from Remark \ref{8rem27} and \eqref{8e26} that $w^{q-1}\in\ca_{q, w_{\ch_{\fz}^{\dz}}}$.
Moreover, by Theorem \ref{prop-0916-1}, there exists $s\in[1, \fz)$ such that $w\in\ca_{s, \dz}$.
Thus, by Lemma \ref{lm-0919-1} with $g:=w^{q-1}$, there exists a positive constant $\gamma\in(0,1)$ such that, for any cube $Q\subset \rn$,
\[
\lf[\fint_{Q}w(x)^{(q-1)(1+\gamma)}\,dw_{\ch_{\infty}^{\delta}}\r]^{\frac{1}{1+\gamma}}\ls \fint_{Q}w(x)^{q-1}\,dw_{\ch_{\infty}^{\delta}}.
\]
This, together with Remark \ref{8rem29}, further implies that
\begin{align}\label{eq0921-1}
\frac{\int_{Q}w(x)^{(q-1)(1+\gamma)+1}\,d\ch_{\infty}^{\delta}}{w_{\ch_{\fz}^{\dz}}(Q)}\ls \lf[\frac{\int_{Q}w(x)^{q}\,d\ch_{\infty}^{\delta}}{w_{\ch_{\fz}^{\dz}}(Q)}\r]^{1+\gamma}.
\end{align}
Let $p:=(q-1)(1+\gamma)+1$. Then $p>q$, and by \eqref{eq0921-1} and $w\in\rh_{q,\dz}$, we find that, for any cube $Q\subset\rn$,
\begin{align*}
&\lf[\sum_{j\in\zz}\ch_{\fz}^{\dz}(  W_j\cap Q)\r]^{1-\frac{1}{p}}\lf[\int_{2Q}w(x)^p\,d\ch_{\fz}^{\dz}\r]^{\frac{1}{p}}\\
&\quad\ls \lf[\sum_{j\in\zz}\ch_{\fz}^{\dz}(  W_j\cap Q)\r]^{1-\frac{1}{p}}\lf[\int_{2Q}w(x)\,d\ch_{\fz}^{\dz}\r]^{-\frac{\gamma}{p}}\lf[\int_{2Q}w(x)^q\,d\ch_{\fz}^{\dz}\r]^{\frac{1+\gamma}{p}}\\
&\quad\ls\lf[\sum_{j\in\zz}\ch_{\fz}^{\dz}(  W_j\cap Q)\r]^{\frac{(q-1)(1+\gamma)}{p}}\lf[\int_{2Q}w(x)^q\,d\ch_{\fz}^{\dz}\r]^{\frac{1+\gamma}{p}}w_{\ch_{\fz}^{\dz}}(Q)^{-\frac{\gamma}{p}}\\
&\quad \ls[w]_{\rh_{q,\dz}}^{\frac{(1+\gamma)q}{p}} w_{\ch_{\fz}^{\dz}}(Q)^{\frac{(1+\gamma)q}{p}}w_{\ch_{\fz}^{\dz}}(Q)^{-\frac{\gamma}{p}}
\sim  [w]_{\rh_{q,\dz}}^{\frac{(1+\gamma)q}{p}} w_{\ch_{\fz}^{\dz}}(Q),
\end{align*}
which implies that $w\in\rh_{p, \dz}$.
This completes the proof of Theorem \ref{lm-0921-1} by taking $\varepsilon:=p-q$.
\end{proof}

\section{Proof of Theorem \ref{them0923-1}}\label{s4}

To prove Theorem \ref{them0923-1}, we need the following useful lemma, which comes from \cite[Lemma 2.6 and Proposition 2.12]{HZZ25}.

\begin{lem}\label{lm925-3}
Let $w\in\ca_{p, \dz}$ with $p\in[1, \fz)$. Assume that $\{Q_j\}_{j\in\nn}$ is a family of non-overlapping dyadic cubes of $\rn$, then there exists a subfamily $\{Q_{j_v}\}_{v}$ and non-overlapping dyadic cubes $\{Q^*_{k}\}_{k}$ satisfying
\begin{enumerate}
\item[{\rm(i)}]
$\bigcup_{j\in\nn}Q_j\subset (\bigcup_{v}Q_{j_v})\cup (\bigcup_{k}Q^*_k$);
\item[{\rm(ii)}]
there exists a positive constant $K$ such that
$\sum_{v}w_{\mathcal H^{\delta}_{\infty}}(Q_{j_v})\le Kw_{\mathcal H^{\delta}_{\infty}}(\bigcup_{v}Q_{j_v})$;
\item[{\rm(iii)}] for any $Q_k^*$, there exists $j\in\nn$ such that $Q_j\subset Q_k^*$, and
$w _{\mathcal H^\delta_\infty}(Q_k^*)\le \sum_{Q_{j_v}\subset Q_k^*} w _{\mathcal H^\delta_\infty}(Q_{j_v}).$
\end{enumerate}
\end{lem}

In the proof of Theorem \ref{prop-0916-1}, we showed that
$
\max\{w,N\}\in \rh_{q,\dz}
$
if $w\in\rh_{q,\dz}$. The following result establishes the analogous property for truncation $w_N:=\min\{w,N\}$. A key feature of this truncation is that the 
function $w_N$ is bounded, which plays an important role in the proof of Theorem \ref{them0923-1}.

\begin{lem}\label{lm-930-1}
Let $w\in\rh_{q, \dz}$ with $q\in(1, \fz]$. Then, for any $N\in\nn$, $w_N:=\min\{w, N\}\in\rh_{q,\dz}$ and the constant involved is independent of $N$.
\end{lem}

\begin{proof}
Notice that for $w_N$, the inequality \eqref{eq831-1} also holds. Combining this with \eqref{eq0902-2}, for any cube $Q\subset\rn$, we have
\begin{align}\label{eq831-3}
\sum_{j\in\zz}\ch_{\fz}^{\dz}(W_{N,j}\cap Q)\leq\sum_{j\in\zz}\ch_{\fz}^{\dz}(w_{j}\cap Q)\ls\ch_{\fz}^{\dz}(Q).
\end{align}
Here $W_{N,j}:=\{x\in\rn: 2^{j-1}<w_N(x)\leq 2^j\}$ and the implicit constant comes from \eqref{eq0902-2}, independent of $N$. Next, we prove that there exists a positive constant $K_1$, independent of $N$, such that for any cube $Q\subset\rn$,
\begin{align}\label{eq831-2}
\int_{2Q}w_N(x)\,d\ch_{\fz}^{\dz}\leq K_1\int_{Q}w_N(x)\,d\ch_{\fz}^{\dz}.
\end{align}
Indeed, for any cube $Q\subset\rn$, set $E_Q:=\{x\in Q: w(x)>N\}$. We verify \eqref{eq831-2} by considering two cases below.

\textbf{Case 1)} $\ch_{\fz}^{\dz}(E_Q)\leq\frac{1}{2}\ch_{\fz}^{\dz}(Q)$. Then
\begin{align}\label{eq831-4}
\ch_{\fz}^{\dz}(Q\setminus E_Q)\ge \ch_{\fz}^{\dz}(Q)-\ch_{\fz}^{\dz}(E_Q)\ge\frac{1}{2}\ch_{\fz}^{\dz}(Q).
\end{align}
In addition, since $w\in\rh_{q,\dz}$, by Theorem \ref{prop-0916-1}, there exists $p\in(1,\fz)$ such that $w\in\ca_{p,\dz}$.
Then it follows from \eqref{eq831-4} and \cite[Lemma 2.4]{HZZ25} that
\[
\int_Qw(x)\,d\ch_{\fz}^{\dz}\ls\int_{Q\setminus E_Q}w(x)\,d\ch_{\fz}^{\dz}\ls\int_Qw_{N}(x)\,d\ch_{\fz}^{\dz}.
\]
This, together with \eqref{eq1111-1}, further implies that
\[
\int_{2Q}w_N(x)\,d\ch_{\fz}^{\dz}\leq\int_{2Q}w(x)\,d\ch_{\fz}^{\dz}\ls\int_Qw(x)\,d\ch_{\fz}^{\dz}\ls\int_Qw_{N}(x)\,d\ch_{\fz}^{\dz}.
\]
Hence \eqref{eq831-2} holds in this case.

\textbf{Case 2)} $\ch_{\fz}^{\dz}(E_Q)>\frac{1}{2}\ch_{\fz}^{\dz}(Q)$. Then
\[
\int_Qw_{N}(x)\,d\ch_{\fz}^{\dz}\ge N\ch_{\fz}^{\dz}(E_Q)>\frac{N}{2}\ch_{\fz}^{\dz}(Q).
\]
Combining this with Remark \ref{12-4}(i), we conclude that
\[
\int_{2Q}w_N(x)\,d\ch_{\fz}^{\dz}\leq N\ch_{\fz}^{\dz}(2Q)=2^{\dz}N\ch_{\fz}^{\dz}(Q)\leq 2^{\dz+1}\int_Qw_N(x)\,d\ch_{\fz}^{\dz}.
\]
Therefore \eqref{eq831-2} also holds in this case.

To prove $w_N\in\rh_{q,\dz}$, in view of \eqref{eq831-3} and \eqref{eq831-2}, it suffices to prove that there exists a positive constant $K_2$, independent of $N$, such that for any cube $Q\subset\rn$,
\begin{align}\label{eq831-5}
\lf(\int_{Q}w_{N}(x)^{q}\,d\ch_{\fz}^{\dz}\r)^{\frac{1}{q}}\leq \frac{K_2}{\ch_{\fz}^{\dz}(Q)^{1-\frac{1}{q}}}\int_Qw_N(x)\,d\ch_{\fz}^{\dz}.
\end{align}
Here and in what follows, for $q=\fz$, the expression $\lf(\int_{E}f(x)^{q}\,d\ch_{\fz}^{\dz}\r)^{\frac{1}{q}}$ stands for $\|f\|_{L^{\fz}(E, \ch_{\fz}^{\dz})}$ and $\frac{1}{q}=0$. For any cube $Q\subset\rn$, we now prove \eqref{eq831-5} by considering two cases.

\textbf{Case 1)} $\int_{E_Q}w(x)\,d\ch_{\fz}^{\dz}\leq \frac{1}{2}\int_Qw(x)\,d\ch_{\fz}^{\dz}$. Then by Remark \ref{rem-11112-1}(i), we have
\[
\int_Qw(x)\,d\ch_{\fz}^{\dz}\leq\int_{Q\setminus E_Q}w(x)\,d\ch_{\fz}^{\dz}+\int_{E_Q}w(x)\,d\ch_{\fz}^{\dz}
\leq \int_{Q\setminus E_Q}w_N(x)\,d\ch_{\fz}^{\dz}+\frac{1}{2}\int_Qw(x)\,d\ch_{\fz}^{\dz}.
\]
That is,
\begin{align}\label{eq831-8}
\int_Qw(x)\,d\ch_{\fz}^{\dz}\leq2\int_{Q\setminus E_Q}w_N(x)\,d\ch_{\fz}^{\dz}\leq 2\int_Qw_N(x)\,d\ch_{\fz}^{\dz}.
\end{align}
Combining this with \eqref{8e26} and \eqref{8e30}, we get
\begin{align*}
\lf(\int_{Q}w_{N}(x)^{q}\,d\ch_{\fz}^{\dz}\r)^{\frac{1}{q}}
&\leq \lf(\int_{Q}w(x)^{q}\,d\ch_{\fz}^{\dz}\r)^{\frac{1}{q}}\ls \frac{1}{\ch_{\fz}^{\dz}(Q)^{1-\frac{1}{q}}}\int_Qw(x)\,d\ch_{\fz}^{\dz}\\
&\ls\frac{1}{\ch_{\fz}^{\dz}(Q)^{1-\frac{1}{q}}}\int_Qw_N(x)\,d\ch_{\fz}^{\dz}.
\end{align*}
Hence \eqref{eq831-5} holds in this case.

\textbf{Case 2)} $\int_{E_Q}w(x)\,d\ch_{\fz}^{\dz}> \frac{1}{2}\int_Qw(x)\,d\ch_{\fz}^{\dz}$. Then, by H\"older's inequality, \eqref{8e26} and \eqref{8e30}, we have
\begin{align*}
\int_Qw(x)\,d\ch_{\fz}^{\dz}&\leq 2\int_{E_Q}w(x)\,d\ch_{\fz}^{\dz}
\ls \lf(\int_{E_Q}w(x)^{q}\,d\ch_{\fz}^{\dz}\r)^{\frac{1}{q}}\ch_{\fz}^{\dz}(E_Q)^{1-\frac1q}\\
&\ls\int_Qw(x)\,d\ch_{\fz}^{\dz}\lf(\frac{\ch_{\fz}^{\dz}(E_Q)}{\ch_{\fz}^{\dz}(Q)} \r)^{1-\frac{1}{q}}.
\end{align*}
This yields $\ch_{\fz}^{\dz}(E_Q)\gs \ch_{\fz}^{\dz}(Q)$. Consequently,
\[
\int_Qw_N(x)\,d\ch_{\fz}^{\dz}\ge \int_{E_Q}w_N(x)\,d\ch_{\fz}^{\dz}=N\ch_{\fz}^{\dz}(E_Q)\gs N\ch_{\fz}^{\dz}(Q).
\]
Namely,
$
N\ls \fint_Qw_N(x)\,d\ch_{\fz}^{\dz}.
$
Therefore,
\begin{align*}
\lf(\int_Qw_N(x)^q\,d\ch_{\fz}^{\dz}\r)^{\frac1q}\leq N\ch_{\fz}^{\dz}(Q)^{\frac1q}\ls\frac{1}{\ch_{\fz}^{\dz}(Q)^{1-\frac{1}{q}}}\int_Qw_N(x)\,d\ch_{\fz}^{\dz}.
\end{align*}
Thus \eqref{eq831-5} also holds in this case, which completes the proof of Lemma \ref{lm-930-1}.
\end{proof}

We next show Theorem \ref{them0923-1}.

\begin{proof}[Proof of Theorem \ref{them0923-1}]
We complete the proof by three steps as below. Our proof strategy follows the ideas in \cite{PA07,PAJMM07}. However, the current framework---beyond Lebesgue measures---introduces new challenges: the Hausdorff content is non-additive and the Choquet integral is nonlinear. 

\textbf{Step 1)} We prove the inequality \eqref{eq923-5}. It suffices to consider the case $H_2=G$. Indeed, for general $H_2$ and $G$, set $\widetilde G=G+H_2$. Then \eqref{eq923-2} holds with $\widetilde G$ in place of $H_2$ and \eqref{eq923-4} holds with $\widetilde G$ in place of $G$. Thus, it remains only to handle the equality case.

For any $\lambda>0$, let
\[E_{\lambda}:=\lf\{x\in\rn: \cm_{\ch_{\fz}^{\dz}}F(x)>\lambda\r\}\cup\lf\{x\in\rn: \cm_{\ch_{\fz}^{\dz}}H_1(x)>\lambda\r\}.\]
By Remark \ref{rem925-1}(ii), $E_{\lambda}$ is open. We may assume without loss of generality that $E_{\lambda}\neq \emptyset$ and $E_{\lambda}\neq \rn$. Then by Whitney's decomposition, there exists a sequence of disjoint cubes $\{Q_j\}_{j\in\nn}$ such that
$E_{\lambda}=\cup_{j\in\nn}Q_j$ and
for every $j\in\nn$, one has $4Q_j\cap E^c_{\lambda}\neq \emptyset$, that is, there exists $x_j\in 4Q_j$ with
\begin{align}\label{eq925-2}
\cm_{\ch_{\fz}^{\dz}}F(x_j)\leq\lambda\quad\text{and} \quad \cm_{\ch_{\fz}^{\dz}}H_1(x_j)\leq\lambda.
\end{align}
Let
\[B_{\lambda}:=\lf\{x\in\rn: \cm_{\ch_{\fz}^{\dz}}F(x)>K\lambda,\ 2G(x)\le \gamma\lambda\r\}.\]
Then, to prove \eqref{eq923-5}, we only need to show
\begin{align}\label{9e02}
w_{\ch_\fz^\delta}(B_{\lambda})
\ls\Bigl(\frac{a^q}{K^q}+\frac{\gamma}{K}\Bigr)^{\frac{1}{s}}w_{\ch_{\fz}^{\dz}}(E_{\lambda}).
\end{align}

For this purpose, we first claim that for any dyadic cube $P$, if there exists $j\in\nn$ such that $Q_j\subset P$, then there exists a positive constant $C$, independent of $P$, such that
\begin{align}\label{eq925-3}
\ch_{\fz}^{\dz}(B_{\lambda}\cap P)\leq C\lf(\frac{a^q}{K^q}+\frac{\gamma}{K}\r)\ch_{\fz}^{\dz}(P).
\end{align}
Applying this and Lemma \ref{lm0922-2} for $w\in\rh_{s', \dz}$, we immediately obtain
\begin{align}\label{eq925-7}
w_{\ch_{\fz}^{\dz}}(B_{\lambda}\cap P)\ls \lf(\frac{a^q}{K^q}+\frac{\gamma}{K}\r)^{\frac1s}w_{\ch_{\fz}^{\dz}}(P).
\end{align}

To prove this claim, we may assume $Q_j\subset P$ and 
$B_{\lambda}\cap P\neq \emptyset$. Then there exists $x^{*}_j\in P$ such that $G(x^*_j)\le \frac{\gamma\lambda}{2}$. From \eqref{eq925-2}, for any cube $Q$ containing $x_j$,
\begin{align}\label{eq925-x}
\fint_{Q}F(x)\,d\ch_{\fz}^{\dz}\leq \lambda.
\end{align}
Choose $K_0\ge C_0$, where $C_0$ is the positive implicit constant in Remark \ref{rem925-1}(iii). For any $x\in B_{\lambda}\cap P$, by Remark \ref{rem925-1}(iii), $K\ge K_0$, and \eqref{eq925-x},
\begin{align}\label{eq925-b}
K\lambda<\cm_{\ch_{\fz}^{\dz}}F(x)\leq C_0\cm^*_{\ch_{\fz}^{\dz}}F(x)=C_0\sup_{Q(x,r)\notni x_j}\fint_{Q(x, r)}F(y)\,d\ch_{\fz}^{\dz}.
\end{align}
Since $x_j\in4Q_j\subset4P$ and $x\in P$, then any cube $Q(x, r)$ not containing $x_j$ satisfies $Q(x,r)\subset 5P$. 
Moreover, \eqref{eq925-b} infers that,
\[K\lambda<C_0\sup_{Q(x,r)\notni x_j}\fint_{Q(x, r)}F(y)\mathbf{1}_{5P}(y)\,d\ch_{\fz}^{\dz}\leq C_0\cm_{\ch_{\fz}^{\dz}}(F\mathbf{1}_{5P})(x),\]
which implies
\[B_{\lambda}\cap P\subset \lf\{x\in\rn: \cm_{\ch_{\fz}^{\dz}}(F\mathbf{1}_{5P})(x)\ge \frac{K\lambda}{C_0}\r\}.\]
Using this, Remark \ref{rem925-1}(iv), and \eqref{eq923-1}, we find that
\begin{align}\label{eq925-1}
\ch_{\fz}^{\dz}(B_{\lambda}\cap P)
&\leq \ch_{\fz}^{\dz}\Bigl(\Bigl\{x\in\rn: \cm_{\ch_{\fz}^{\dz}}(F\mathbf{1}_{5P})(x)\ge \frac{K\lambda}{C_0}\Bigr\}\Bigr)\\
&\leq \ch_{\fz}^{\dz}\Bigl(\Bigl\{x\in\rn: \cm_{\ch_{\fz}^{\dz}}(G_{5P}\mathbf{1}_{5P})(x)\ge \frac{K\lambda}{4C_0}\Bigr\}\Bigr)\noz\\
&\quad+\ch_{\fz}^{\dz}\Bigl(\Bigl\{x\in\rn: \cm_{\ch_{\fz}^{\dz}}(H_{5P}\mathbf{1}_{5P})(x)\ge \frac{K\lambda}{4C_0}\Bigr\}\Bigr)=: I_1+I_2.\noz
\end{align}
For $I_1$, from Remark \ref{rem925-1}(v) with $p=1$, \eqref{eq923-4}, and $x_j^*\in P\subset 5P$, we deduce that
\begin{align}\label{eq925-5}
I_1\ls \frac{1}{K\lambda}\int_{5P}G_{5P}(x)\,d\ch_{\fz}^{\dz}\ls  \frac{1}{K\lambda}\ch_{\fz}^{\dz}(5P)G(x_j^*)\ls \frac{\gamma}{K}\ch_{\fz}^{\dz}(P).
\end{align}
For $I_2$, by Remark \ref{rem925-1}(v) with $p:=q\in(1, \fz)$, \eqref{eq923-2}, \eqref{eq925-2}, $\gamma\in(0, 1)$, and $x_j, x_j^*\in 5P$,
\begin{align}\label{eq925-4}
I_2
&\ls\frac{1}{(K\lambda)^q}\int_{5P}H_{5P}(x)^q\,d\ch_{\fz}^{\dz}\\
&\ls \Bigl(\frac{a}{K\lambda}\Bigr)^q\ch_{\fz}^{\dz}(5P)\lf[\cm_{\ch_{\fz}^{\dz}}F(x_j)+\cm_{\ch_{\fz}^{\dz}}H_1(x_j)+G(x_j^*)\r]^q\ls\frac{a^q}{K^q}\ch_{\fz}^{\dz}(P).\noz
\end{align}
When $q=\fz$, by Remark \ref{rem925-1}(vi), \eqref{eq923-2}, $\gamma\in(0, 1)$, and $x_j, x_j^*\in 5P$, we have
\[\lf\|\cm_{\ch_{\fz}^{\dz}}(H_{5P}\mathbf{1}_{5P})\r\|_{L^{\fz}(\rn, \ch_{\fz}^{\dz})}\leq\lf\|H_{5P}\r\|_{L^{\fz}(5P, \ch_{\fz}^{\dz})}\leq a\lf[\cm_{\ch_{\fz}^{\dz}}F(x_j)+\cm_{\ch_{\fz}^{\dz}}H_1(x_j)+H_2(x_j^*)\r]\leq 3a\lambda.\]
Choosing $K_0=12aC_0$ gives $I_2=0$.
Combining \eqref{eq925-1}, \eqref{eq925-5}, and \eqref{eq925-4}, we conclude that \eqref{eq925-3} holds, and hence the claim is valid.

From Theorem \ref{prop-0916-1}, there exists $r\in[1, \fz)$ such that $w\in\ca_{r, \dz}$. Thus, applying Lemma \ref{lm925-3} to $\{Q_j\}_{j\in\nn}$, we obtain a subfamily $\{Q_{j_v}\}_{v}$ and non-overlapping dyadic cubes $\{Q^*_{k}\}_{k}$ such that
\begin{enumerate}
\item[{\rm(i)}]
$\bigcup_{j\in\nn}Q_j\subset (\bigcup_{v}Q_{j_v})\cup (\bigcup_{k}Q^*_k);$
\item[{\rm(ii)}]
there exists a positive constant $K$ with
$\sum_{v}w_{\mathcal H^{\delta}_{\infty}}(Q_{j_v})\le Kw_{\mathcal H^{\delta}_{\infty}}\Bigl(\bigcup_{v}Q_{j_v}\Bigr);$
\item[{\rm(iii)}]
for every $Q_k^*$, there exists $j\in\nn$ such that $Q_j\subset Q_k^*$ and
$w _{\mathcal H^\delta_\infty}(Q_k^*)\le \sum_{Q_{j_v}\subset Q_k^*} w _{\mathcal H^\delta_\infty}(Q_{j_v}).$
\end{enumerate}
Thus, by the above claim, \eqref{eq925-7} holds for all $Q_{j_v}$ and $Q_k^*$.

Note that $K\ge K_0>a\ge 1$. Thus $B_{\lambda}\subset E_{\lambda}$. Then by (i), we obtain
\[B_{\lambda}\subset \bigcup_{j\in\nn}(B_{\lambda}\cap Q_j)\subset \bigcup_{v}(B_{\lambda}\cap Q_{j_v})\cup \bigcup_{k}(B_{\lambda}\cap Q^*_k).\]
This, together with \eqref{eq925-7} for $Q_{j_v}$ and $Q_k^*$, (iii), and (ii), further implies that
\begin{align*}
w_{\ch_\fz^\delta}(B_{\lambda})
&\leq \sum_{v}w_{\ch_\fz^\delta}(B_{\lambda}\cap Q_{j_v})+\sum_{k}w_{\ch_\fz^\delta}(B_{\lambda}\cap Q^*_k)\\
&\ls \Bigl(\frac{a^q}{K^q}+\frac{\gamma}{K}\Bigr)^{\frac1s}\lf[\sum_{v}w_{\ch_\fz^\delta}(Q_{j_v})+\sum_{k}w_{\ch_\fz^\delta}(Q^*_k)\r]\\
&\ls \Bigl(\frac{a^q}{K^q}+\frac{\gamma}{K}\Bigr)^{\frac1s}\lf[\sum_{v}w_{\ch_\fz^\delta}(Q_{j_v})+\sum_{k}\sum_{Q_{j_v}\subset Q_k^*}w_{\ch_\fz^\delta}(Q_{j_v})\r]\\
&\ls \Bigl(\frac{a^q}{K^q}+\frac{\gamma}{K}\Bigr)^{\frac1s}\sum_{v}w_{\ch_\fz^\delta}(Q_{j_v})\ls \Bigl(\frac{a^q}{K^q}+\frac{\gamma}{K}\Bigr)^{\frac1s}w_{\ch_\fz^\delta}\Bigl(\bigcup_v Q_{j_v}\Bigr)\\
&\ls \Bigl(\frac{a^q}{K^q}+\frac{\gamma}{K}\Bigr)^{\frac1s}w_{\ch_\fz^\delta}(E_{\lambda}),
\end{align*}
which proves \eqref{9e02}, and hence of \eqref{eq923-5}.

\textbf{Step 2)} We now show \eqref{eq926-1} and \eqref{eq926-2} via \eqref{eq923-5}. For any $K\ge K_0$, $\lambda\in (0, \fz)$, $\gamma\in(0, 1)$, by \eqref{eq923-5}, we have
\begin{align}\label{eq926-4}
&w_{\ch_{\fz}^{\dz}}\bigl\{x\in\rn: \cm_{\ch_{\fz}^{\dz}}F(x)>K\lambda\bigr\}\\
&\quad\leq w_{\ch_{\fz}^{\dz}}\bigl\{x\in\rn: \cm_{\ch_{\fz}^{\dz}}F(x)>K\lambda,\, G(x)+H_2(x)\leq \gamma\lambda\bigr\}\noz\\
&\quad\quad+w_{\ch_{\fz}^{\dz}}\bigl\{x\in\rn: G(x)+H_2(x)>\gamma\lambda\bigr\}\noz\\
&\quad\leq C\Bigl(\frac{a^q}{K^q}+\frac{\gamma}{K}\Bigr)^{\frac{1}{s}}w_{\ch_{\fz}^{\dz}}\bigl\{x\in\rn: \cm_{\ch_{\fz}^{\dz}}F(x)+\cm_{\ch_{\fz}^{\dz}}H_1(x)>\lambda\bigr\}\noz\\
&\quad\quad+w_{\ch_{\fz}^{\dz}}\bigl\{x\in\rn: G(x)+H_2(x)>\gamma\lambda\bigr\},\noz
\end{align}
where $C$ is as in \eqref{eq923-5}. This, combined with \eqref{eq1112-3}, Remark \ref{rem-11112-1}(iii), and Remark \ref{8rem29}, implies that
\begin{align}\label{eq926-3}
&\int_{\rn}\bigl|\cm_{\ch_{\fz}^{\dz}}F(x)\bigr|^pw(x)\,d\ch_{\fz}^{\dz}\\
&\quad=4pK^p\int_{0}^{\fz}\lambda^{p-1}w_{\ch_{\fz}^{\dz}}\bigl\{x\in\rn: \cm_{\ch_{\fz}^{\dz}}F(x)>K\lambda\bigr\}\,d\lambda\noz\\
&\quad\le4pCK^p\Bigl(\frac{a^q}{K^q}+\frac{\gamma}{K}\Bigr)^{\frac{1}{s}}\int_{0}^{\fz}\lambda^{p-1}w_{\ch_{\fz}^{\dz}}\bigl\{x\in\rn: \cm_{\ch_{\fz}^{\dz}}F(x)+\cm_{\ch_{\fz}^{\dz}}H_1(x)>\lambda\bigr\}\,d\lambda\noz\\
&\quad\quad+4pK^p\int_{0}^{\fz}\lambda^{p-1}w_{\ch_{\fz}^{\dz}}\bigl\{x\in\rn: G(x)+H_2(x)>\gamma\lambda\bigr\}\,d\lambda\noz\\
&\quad=4CK^p\Bigl(\frac{a^q}{K^q}+\frac{\gamma}{K}\Bigr)^{\frac{1}{s}}\int_{\rn}\lf[\cm_{\ch_{\fz}^{\dz}}F(x)+\cm_{\ch_{\fz}^{\dz}}H_1(x)\r]^p\,dw_{\ch_{\fz}^{\dz}}\noz\\
&\quad\quad+\frac{4K^p}{\gamma^p}\int_{\rn}\lf[G(x)+H_2(x)\r]^p\,dw_{\ch_{\fz}^{\dz}}\noz\\
&\quad\leq CW_0K^p(2^{p+2}+8)\Bigl(\frac{a^q}{K^q}+\frac{\gamma}{K}\Bigr)^{\frac{1}{s}}\lf(\|\cm_{\ch_{\fz}^{\dz}}F\|^p_{L^p_w(\rn, \ch_{\fz}^{\dz})}+\|\cm_{\ch_{\fz}^{\dz}}H_1\|^p_{L^p_w(\rn, \ch_{\fz}^{\dz})}\r)\noz\\
&\quad\quad+\frac{W_0K^p}{\gamma^p}(2^{p+2}+8)\lf(\|G\|^p_{L^p_w(\rn, \ch_{\fz}^{\dz})}+\|H_2\|^p_{L^p_w(\rn, \ch_{\fz}^{\dz})}\r),\noz
\end{align}
where $W_0$ is as in Remark \ref{8rem29}.
Note that
\[CW_0K^p(2^{p+2}+8)\Bigl(\frac{a^q}{K^q}+\frac{\gamma}{K}\Bigr)^{\frac{1}{s}}\leq CW_0(2^{p+2}+8)a^{\frac{q}{s}}K^{p-\frac{q}{s}}+CW_0(2^{p+2}+8)\gamma^{\frac1s}K^{p-\frac1s}\]
and $p\in(0,\frac{q}{s})$. Choose $K$ large enough such that
\[CW_0(2^{p+2}+8)a^{\frac{q}{s}}K^{p-\frac{q}{s}}\le \frac14.\]
Fixing this $K$, choose $\gamma$ small enough such that
\[CW_0(2^{p+2}+8)\gamma^{\frac1s}K^{p-\frac1s}\leq \frac14.\]
Then
\[CW_0K^p(2^{p+2}+8)\Bigl(\frac{a^q}{K^q}+\frac{\gamma}{K}\Bigr)^{\frac{1}{s}}\leq \frac12.\]
Then by the assumption $\|\cm_{\ch_{\fz}^{\dz}}F\|_{L_{w}^{p}(\rn, \ch_{\fz}^{\dz})}<\fz$ and rearranging \eqref{eq926-3}, we obtain \eqref{eq926-1}. Similarly, \eqref{eq926-2} also follows from \eqref{eq926-4}.

\textbf{Step 3)} We
finally prove \eqref{eq926-1} and \eqref{eq926-2} for $p\ge 1$ and $F\in L^{1}(\rn, \ch_{\fz}^{\dz})$. It suffices to show that $\cm_{\ch_{\fz}^{\dz}}F\in L^p_w(\rn, \ch_{\fz}^{\dz})$ by Step 2).
To this end, we may assume without loss of generality that $\cm_{\ch_{\fz}^{\dz}}H_1$, $G$, $H_2\in L^p_w(\rn, \ch_{\fz}^{\dz})$. We first show that $\cm_{\ch_{\fz}^{\dz}}F\in L^p_w(\rn, \ch_{\fz}^{\dz})$ when $w\in L^{\fz}(\rn, \ch_{\fz}^{\dz})$. For any $\lambda>0$, let
\[f(\lambda):=p\lambda^pw_{\ch_{\fz}^{\dz}}\bigl(\bigl\{x\in\rn: \cm_{\ch_{\fz}^{\dz}}F(x)>\lambda\bigr\}\bigr).\]
For any $0<\lambda_{1}<\lambda_2<\fz$, by Remark \ref{rem925-1}(v) with $p=1$,
\begin{align}\label{eq101-1}
\int_{\lambda_1}^{\lambda_2}f(\lambda)\,\frac{d\lambda}{\lambda}
&\leq \frac{p\lambda_2^p(\lambda_2-\lambda_1)}{\lambda_1}w_{\ch_{\fz}^{\dz}}\bigl(\bigl\{x\in\rn: \cm_{\ch_{\fz}^{\dz}}F(x)>\lambda_1\bigr\}\bigr)\\
&\leq \frac{p\lambda_2^p(\lambda_2-\lambda_1)}{\lambda_1}\|w\|_{L^{\fz}(\rn, \ch_{\fz}^{\dz})}\ch_{\fz}^{\dz}\bigl(\bigl\{x\in\rn: \cm_{\ch_{\fz}^{\dz}}F(x)>\lambda_1\bigr\}\bigr)\noz\\
&\ls \frac{p\lambda_2^p(\lambda_2-\lambda_1)}{\lambda_1^2}\|w\|_{L^{\fz}(\rn, \ch_{\fz}^{\dz})}\|F\|_{L^{1}(\rn, \ch_{\fz}^{\dz})}<\fz.\noz
\end{align}
By \eqref{eq923-5}, choosing large $K$ and small $\gamma$ such that
\[C2^pK^p\Bigl(\frac{a^q}{K^q}+\frac{\gamma}{K}\Bigr)^{\frac{1}{s}}\leq \frac12,\]
and Remark \ref{8rem29}, we obtain
\begin{align}\label{eq930-1}
&\int_{\lambda_1}^{\lambda_2}f(\lambda)\,\frac{d\lambda}{\lambda}=\int_{\frac{\lambda_1}{K}}^{\frac{\lambda_2}{K}}f(K\lambda)\,\frac{d\lambda}{\lambda}\\
&\quad\leq CK^p\Bigl(\frac{a^q}{K^q}+\frac{\gamma}{K}\Bigr)^{\frac{1}{s}}\int_{\frac{\lambda_1}{K}}^{\frac{\lambda_2}{K}}p\lambda^{p-1} w_{\ch_{\fz}^{\dz}}\bigl\{x\in\rn: \cm_{\ch_{\fz}^{\dz}}F(x)+\cm_{\ch_{\fz}^{\dz}}H_1(x)>\lambda\bigr\}\,d\lambda\noz\\
&\quad\quad+K^p\int_{\frac{\lambda_1}{K}}^{\frac{\lambda_2}{K}}p\lambda^{p-1} w_{\ch_{\fz}^{\dz}}\bigl\{x\in\rn: G(x)+H_2(x)>\gamma\lambda\bigr\}\,d\lambda\noz\\
&\quad\leq CK^p\Bigl(\frac{a^q}{K^q}+\frac{\gamma}{K}\Bigr)^{\frac{1}{s}}\int_{\frac{\lambda_1}{K}}^{\frac{\lambda_2}{K}}p\lambda^{p-1} w_{\ch_{\fz}^{\dz}}\lf\{x\in\rn: \cm_{\ch_{\fz}^{\dz}}F(x)>\frac{\lambda}{2}\r\}\,d\lambda\noz\\
&\quad\quad+CK^p\Bigl(\frac{a^q}{K^q}+\frac{\gamma}{K}\Bigr)^{\frac{1}{s}}\int_{\frac{\lambda_1}{K}}^{\frac{\lambda_2}{K}}p\lambda^{p-1} w_{\ch_{\fz}^{\dz}}\lf\{x\in\rn: \cm_{\ch_{\fz}^{\dz}}H_1(x)>\frac{\lambda}{2}\r\}\,d\lambda\noz\\
&\quad\quad+K^p\int_{\frac{\lambda_1}{K}}^{\frac{\lambda_2}{K}}p\lambda^{p-1} w_{\ch_{\fz}^{\dz}}\lf\{x\in\rn: G(x)>\frac{\gamma\lambda}{2}\r\}\,d\lambda\noz\\
&\quad\quad+K^p\int_{\frac{\lambda_1}{K}}^{\frac{\lambda_2}{K}}p\lambda^{p-1} w_{\ch_{\fz}^{\dz}}\lf\{x\in\rn: H_2(x)>\frac{\gamma\lambda}{2}\r\}\,d\lambda\noz\\
&\quad\leq \frac12\int_{\frac{\lambda_1}{2K}}^{\frac{\lambda_2}{2K}}f(\lambda)\,\frac{d\lambda}{\lambda}+S\noz
\end{align}
where
\[S:=\frac{W_0}2\|\cm_{\ch_{\fz}^{\dz}}H_1\|^{p}_{L^p_w(\rn, \ch_{\fz}^{\dz})}+\frac{2^p{W_0}K^p}{\gamma^p}\lf(\|G\|^{p}_{L^p_w(\rn, \ch_{\fz}^{\dz})}+\|H_2\|^{p}_{L^p_w(\rn, \ch_{\fz}^{\dz})}\r)<\fz.\]
Taking $\lambda_1=K^{-n}$ and $\lambda_{2}=K^{m}$ for any $n,m\in\nn$, \eqref{eq930-1} gives
\begin{align*}
\int_{K^{-n}}^{K^{m-1}}f(\lambda)\,\frac{d\lambda}{\lambda}\leq \int_{K^{-n}}^{K^{m}}f(\lambda)\,\frac{d\lambda}{\lambda}\leq \frac12\int_{\frac{K^{-n-1}}{2}}^{K^{m-1}}f(\lambda)\,\frac{d\lambda}{\lambda}+S.
\end{align*}
Using this, $K\ge1$, and \eqref{eq101-1}, we have
\[\int_{K^{-n}}^{K^{m-1}}f(\lambda)\,\frac{d\lambda}{\lambda}\leq \int_{\frac{K^{-n-1}}{2}}^{K^{-n}}f(\lambda)\,\frac{d\lambda}{\lambda}+2S\leq K'\|w\|_{L^{\fz}}\|F\|_{L^{1}(\rn, \ch_{\fz}^{\dz})}+2S<\fz,\]
with $K'$ independent of $m,n$. Hence it follows from Remark \ref{8rem29} that
\[\|\cm_{\ch_{\fz}^{\dz}}F\|^p_{L^p_w(\rn, \ch_{\fz}^{\dz})}\le4\lim_{n, m\to \fz}\int_{K^{-n}}^{K^{m-1}}f(\lambda)\,\frac{d\lambda}{\lambda}<\fz.\]
Thus, $\cm_{\ch_{\fz}^{\dz}}F\in L^p_w(\rn, \ch_{\fz}^{\dz})$ and hence $\cm_{\ch_{\fz}^{\dz}}F\in L^{p, \fz}_w(\rn, \ch_{\fz}^{\dz})$.

Finally, we drop the condition $w\in L^{\fz}(\rn, \ch_{\fz}^{\dz})$.
For any $N\in \nn$, let $w_N:=\min\{w, N\}$. By Lemma \ref{lm-930-1}, inequalities \eqref{eq926-1} and \eqref{eq926-2} hold with $w_N$ in place of $w$, with constants independent of $N$. Since $w_N\uparrow w$ as $N\to \fz$, we complete the proof by (ii) and (iv) of Remark \ref{12-4}.
\end{proof}

\section{Applications}\label{s5}

In this section, we show some direct applications of Theorem \ref{them0923-1}. We begin with the following weighted Fefferman--Stein inequality with Hausdorff content.
Let $f\in L_{\loc}^1(\rn, \ch_{\fz}^{\dz})$. For any $x\in\rn$, define the \emph{sharp maximal operator} $\cm^{\sharp}_{\ch_{\fz}^{\dz}}$ as
\[\cm^{\sharp}_{\ch_{\fz}^{\dz}}f(x):=\sup_{Q\ni x}\inf_{c\in\mathbb R}\fint_{Q}|f(y)-c|\,d\ch_{\fz}^{\dz}.\]

\begin{thm}\label{tm101-1}
(Weighted Fefferman--Stein inequality with Hausdorff content) Let $w\in\ca_{\fz, \dz}$ and $p\in(0, \fz)$. Then there exists a positive constant $C$ such that whenever $\cm_{\ch_{\fz}^{\dz}}f\in L^p_w(\rn, \ch_{\fz}^{\dz})$,
\begin{align}\label{eq101-2}
\|\cm_{\ch_{\fz}^{\dz}}f\|_{L^p_w(\rn, \ch_{\fz}^{\dz})}\leq C\|\cm_{\ch_{\fz}^{\dz}}^{\sharp}f\|_{L^p_w(\rn, \ch_{\fz}^{\dz})}.
\end{align}
\end{thm}

\begin{proof}
For any cube $Q\subset\rn$, by Remark \ref{12-4}(ii) and \cite[Lemma 3.1]{ZZH-2} with $w\equiv 1$, there exists a constant $C_Q\in\mathbb R$ such that
\begin{align}\label{eq101-3}
\fint_Q|f(y)-C_Q|\,d\ch_{\fz}^{\dz}\leq 2\inf_{c\in\mathbb R}\fint_Q|f(y)-c|\,d\ch_{\fz}^{\dz}.
\end{align}
For any $x\in\rn$, set
$F(x):=|f(x)|$, $G(x):=2\cm_{\ch_{\fz}^{\dz}}^{\sharp}f(x)$, and $H_1(x)=H_2(x)\equiv 0$.
Then, for any cube $Q$, the non-negative functions
\[G_Q(x):=|f(x)-C_Q|\mathbf{1}_Q(x)\quad\text{and}\quad H_Q(x):=|C_Q|\mathbf{1}_Q(x)\]
satisfy:
\begin{enumerate}
\item[{\rm(i)}]
$F(x)=|f(x)|\leq |f(x)-C_Q|+ |C_Q|\leq G_Q(x)+ H_Q(x),\ \forall\ x\in Q;$
\item[{\rm(ii)}]
applying Remark \ref{12-4}(iii) and \eqref{eq101-3}, we have, for any $x,\bar{x}\in Q$, 
\begin{align*}
\|H_Q\|_{L^{\fz}(Q, \ch_{\fz}^{\dz})}
&=|C_Q|\leq 2\fint_Q|f(y)-C_Q|\,d\ch_{\fz}^{\dz}+2\fint_Q|f(y)|\,d\ch_{\fz}^{\dz}\\
&\leq4\inf_{c\in\mathbb R}\fint_Q|f(y)-c|\,d\ch_{\fz}^{\dz}+2\fint_Q|f(y)|\,d\ch_{\fz}^{\dz}\\
&\leq 6\lf[\cm_{\ch_{\fz}^{\dz}}F(x)+\cm_{\ch_{\fz}^{\dz}}H_1(x)+H_2(\bar{x})\r];
\end{align*}
\item[{\rm(iii)}]
by \eqref{eq101-3} again, for any $x\in Q$, we find that
\begin{align*}
\fint_{Q}G_Q(y)\,d\ch_{\fz}^{\dz}\leq 2\inf_{c\in\mathbb R}\fint_Q|f(y)-c|\,d\ch_{\fz}^{\dz}\leq G(x).
\end{align*}
\end{enumerate}
Thus, the conditions \eqref{eq923-1}, \eqref{eq923-2} for $q=\fz$, and \eqref{eq923-4} in Theorem \ref{them0923-1} hold.
By Theorem \ref{prop-0916-1}, there exists $s\in [1,\fz)$ such that $w\in\rh_{s', \dz}$. Applying \eqref{eq926-1} of Theorem \ref{them0923-1}, we deduce that, for any $p\in(0,\fz)$,
\[\|\cm_{\ch_{\fz}^{\dz}}f\|_{L^p_w(\rn, \ch_{\fz}^{\dz})}=\|\cm_{\ch_{\fz}^{\dz}}F\|_{L^p_w(\rn, \ch_{\fz}^{\dz})}\ls \|G\|_{L^p_w(\rn, \ch_{\fz}^{\dz})}\sim\|\cm_{\ch_{\fz}^{\dz}}^{\sharp}f\|_{L^p_w(\rn, \ch_{\fz}^{\dz})}.\]
This completes the proof of Theorem \ref{tm101-1}.
\end{proof}

Recall that \cite[Corollary 4.4]{HZZ25} yields that for $\ch_{\fz}^{\dz}$-a.e. $x\in\rn,$
$|f(x)|\ls \cm_{\ch_{\fz}^{\dz}}f(x).$ This, together with Theorem \ref{tm101-1}, implies the following corollary.

\begin{cor}\label{tm101-1}
Let $w\in\ca_{\fz, \dz}$ and $p\in(0, \fz)$. Then there exists a positive constant $C$ such that whenever $f\in L^p_w(\rn, \ch_{\fz}^{\dz})$,
\begin{align}\label{9e5}
\|f\|_{L^p_w(\rn, \ch_{\fz}^{\dz})}\leq C\|\cm_{\ch_{\fz}^{\dz}}^{\sharp}f\|_{L^p_w(\rn, \ch_{\fz}^{\dz})}.
\end{align}
\end{cor}

In what follows, we say that an operator $T$ acts from $A$ to $B$ (where $A$ and $B$ are given sets) if $T$ is a map defined on $A$ and taking values in $B$. An operator $T$ acting from $A$ to $B$, both vector spaces of functions, is sublinear if
\[|T(f+g)|\leq |Tf|+|Tg|,\quad |T(\lambda f)|=|\lambda|\,|Tf|\]
for all $f,g\in A$ and $\lambda\in\mathbb R$ or $\mathbb C$.

\begin{thm}\label{tm102-1}
Let $1\leq p_0<q_0\leq \fz$. Let $\mathcal X$ and $\mathcal Y$ be vector spaces such that $\mathcal{X}\subset \mathcal{Y}$. Let $T$ and $S$ be operators such that $S$ acts from $\mathcal X$ into the space of functions, and $T$ is sublinear and acts from $\mathcal Y$ into $L^{p_0}(\rn, \ch_{\fz}^{\dz})$. Let $\{\mathcal B_{r}\}_{r>0}$ be a family of operators acting from $\mathcal X$ into $\mathcal Y$. Suppose there exists a positive constant $C$ such that for any $f\in\mathcal X$, any cube $Q\subset\rn$, and any $x\in Q$,
\begin{align}\label{eq102-3}
\lf[\fint_Q|T(I-\mathcal B_{l(Q)})f(y)|^{p_0}\,d\ch_{\fz}^{\dz}\r]^{\frac{1}{p_0}}\leq C\,\cm_{\ch_{\fz}^{\dz}}(|Sf|^{p_0})^{\frac{1}{p_0}}(x)
\end{align}
and
\begin{align}\label{eq102-4}
\lf[\fint_Q|T\mathcal B_{l(Q)}f(y)|^{q_0}\,d\ch_{\fz}^{\dz}\r]^{\frac{1}{q_0}}\leq C\,\cm_{\ch_{\fz}^{\dz}}(|Tf|^{p_0})^{\frac{1}{p_0}}(x),
\end{align}
where $l(Q)$ denotes the side length of $Q$, and $I$ is the identity operator. Then, whenever $p_0<p< q_0$ (or $p=q_0$ when $q_0<\fz$) and $w\in \ca_{p/p_0, \dz}\cap \rh_{(q_0/p)',\dz}$, there exists a positive constant $K$ such that for all $f\in \mathcal X$,
\begin{align}\label{eq102-5}
\|Tf\|_{L^p_{w}(\rn, \ch_{\fz}^{\dz})}\leq K\|Sf\|_{L^p_{w}(\rn, \ch_{\fz}^{\dz})}.
\end{align}
\end{thm}

\begin{proof}
We first prove \eqref{eq102-5} in the case $q_0<\fz$ and $p_0<p\leq q_0$. Let $f\in \mathcal X$ such that $F:=|Tf|^{p_0}\in L^1_w(\rn, \ch_{\fz}^{\dz})$. For any $x\in\rn$, define
\[
G(x):=2^{p_0-1}C^{p_0}\cm_{\ch_{\fz}^{\dz}}(|Sf|^{p_0})(x)\quad\text{and}\quad H_1\equiv H_2\equiv0.
\]
Then for any cube $Q\subset\rn$, the non-negative functions
\[
G_Q(x):=2^{p_0-1}\bigl|T(I-\mathcal{B}_{l(Q)})f\bigr|^{p_0}(x)\mathbf{1}_Q(x)\quad\text{and}\quad
H_Q(x):=2^{p_0-1}\bigl|T\mathcal{B}_{l(Q)}f\bigr|^{p_0}(x)\mathbf{1}_Q(x)
\]
satisfy:
\begin{enumerate}
\item[{\rm(i)}]
for any $x\in Q$, by sublinearity of $T$,
\begin{align*}
F(x)\leq 2^{p_0-1}\bigl|T(I-\mathcal{B}_{l(Q)})f\bigr|^{p_0}(x)+ 2^{p_0-1}\bigl|T\mathcal{B}_{l(Q)}f\bigr|^{p_0}(x)= G_Q(x)+ H_Q(x);
\end{align*}
\item[{\rm(ii)}]
by \eqref{eq102-4}, for any $x,\bar{x}\in Q$,
\begin{align*}
\lf[\fint_{Q}|H_Q(y)|^{\frac{q_0}{p_0}}\,d\ch_{\fz}^{\dz}\r]^{\frac{p_0}{q_0}}
&\leq
2^{p_0-1}C^{p_0}\cm_{\ch_{\fz}^{\dz}}(|Tf|^{p_0})(x)\\
&\leq 2^{p_0-1}C^{p_0}\lf[\cm_{\ch_{\fz}^{\dz}}F(x)+\cm_{\ch_{\fz}^{\dz}}H_1(x)+H_2(\bar{x})\r];
\end{align*}
\item[{\rm(iii)}]
by \eqref{eq102-3}, for any $x\in Q$, $
\fint_{Q}G_Q(y)\,d\ch_{\fz}^{\dz}
\leq G(x).$
\end{enumerate}
Thus, the conditions \eqref{eq923-1}, \eqref{eq923-2}, and \eqref{eq923-4} in Theorem \ref{them0923-1} hold.
Furthermore, if $q_0<\fz$, by Theorem \ref{lm-0921-1}, there exists $s\in (p, q_0)$ such that $w\in \ca_{p/p_0, \dz}\cap \rh_{(s/p)',\dz}$. Note that $\frac{p}{p_0}\in (0, \frac{pq_0}{sp_0})$. Therefore, by \cite[Corollary 4.4]{HZZ25}, \eqref{eq926-1} of Theorem \ref{them0923-1}, and \cite[Theorem 1.1]{HZZ25},
\begin{align*}
\|Tf\|^{p_0}_{L^p_{w}(\rn, \ch_{\fz}^{\dz})}&\ls \|\cm_{\ch_{\fz}^{\dz}}F\|_{L^{\frac{p}{p_0}}_{w}(\rn, \ch_{\fz}^{\dz})}\ls \|G\|_{L^{\frac{p}{p_0}}_{w}(\rn, \ch_{\fz}^{\dz})}\ls \|Sf\|^{p_0}_{L^p_{w}(\rn, \ch_{\fz}^{\dz})},
\end{align*}
which completes the proof of Theorem \ref{tm102-1} in this case.

If $q_0=\fz$, then $w\in \ca_{p/p_0, \dz}\cap \rh_{1,\dz}$, and there exists $s\in(1, \fz)$ such that $w\in \ca_{p/p_0, \dz}\cap \rh_{s',\dz}$. A similar argument shows that \eqref{eq102-5} holds for $q_0=\fz$, and hecce Theorem \ref{tm102-1} is proved.
\end{proof}

The following is an extension of Theorem \ref{tm102-1}. for simplicity we assume that $S=I$.

\begin{thm}\label{tm103-1}
Let $1\leq p_0<q_0\leq \fz$. Let $\mathcal X$, $\mathcal Y$, $T$, and $\{\mathcal B_{r}\}_{r>0}$ be as in Theorem \ref{tm102-1}. Assume that \eqref{eq102-3} holds with $S=I$, and, instead of \eqref{eq102-4},  that
\begin{align}\label{eq103-2}
\lf[\fint_Q|T\mathcal B_{l(Q)}f(y)|^{q_0}\,d\ch_{\fz}^{\dz}\r]^{\frac{1}{q_0}}\leq C\lf[\cm_{\ch_{\fz}^{\dz}}(|Tf|^{p_0})^{\frac{1}{p_0}}(x)+\cm_{\ch_{\fz}^{\dz}}(|S_1f|^{p_0})^{\frac{1}{p_0}}(x)+|S_2f(\bar x)|\r]
\end{align}
holds for all $f\in\mathcal X$ and all $x,\bar x\in Q$, where $S_1$ and $S_2$ are given operators. Let $p_0<p<q_0$ and $w\in \ca_{p/p_0, \dz}\cap \rh_{(q_0/p)',\dz}$. If $S_1$ and $S_2$ are bounded on $L^p_w(\rn,\ch_{\fz}^{\dz})$, then there exists a positive constant $C$ such that, for all $f\in\mathcal X\cap L^p_w(\rn, \ch_{\fz}^{\dz})$,
\[
\|Tf\|_{L^p_{w}(\rn, \ch_{\fz}^{\dz})}\leq C\|f\|_{L^p_{w}(\rn, \ch_{\fz}^{\dz})}.
\]

\end{thm}

\begin{proof}
The proof is similar to that of Theorem \ref{tm102-1}. Let $f\in\mathcal X\cap L^p_w(\rn, \ch_{\fz}^{\dz})$ and set
\[
F:=|Tf|^{p_0},\quad H_1:=|S_1f|^{p_0},\quad H_2:=|S_2f|^{p_0}.
\]
Then the conditions \eqref{eq923-1}, \eqref{eq923-2}, and \eqref{eq923-4} in Theorem \ref{them0923-1} hold.
By \cite[Corollary 4.4]{HZZ25}, Remark \ref{12-4}(ii), \eqref{eq926-1} of Theorem \ref{them0923-1}, \cite[Theorem 1.1]{HZZ25}, and the boundedness of $S_1$ and $S_2$ on $L^p_w(\rn,\ch_{\fz}^{\dz})$,
we conclude that
\begin{align*}
\|Tf\|^{p_0}_{L^p_{w}(\rn, \ch_{\fz}^{\dz})}
&\ls \|\cm_{\ch_{\fz}^{\dz}}F\|_{L^{\frac{p}{p_0}}_{w}(\rn, \ch_{\fz}^{\dz})}\\
&\ls\|\cm_{\ch_{\fz}^{\dz}}(|f|^{p_0})\|_{L^{\frac{p}{p_0}}_{w}(\rn, \ch_{\fz}^{\dz})}+\|\cm_{\ch_{\fz}^{\dz}}(|S_1f|^{p_0})\|_{L^{\frac{p}{p_0}}_{w}(\rn, \ch_{\fz}^{\dz})}+\bigl\||S_2f|^{p_0}\bigr\|_{L^{\frac{p}{p_0}}_{w}(\rn, \ch_{\fz}^{\dz})}\\
&\ls \|f\|^{p_0}_{L^{p}_{w}(\rn, \ch_{\fz}^{\dz})},
\end{align*}
which completes the proof of Theorem \ref{tm103-1}.
\end{proof}

The last application of the main theorem gives an extension of Shen \cite[Theorem 3.1]{ZS05}.

\begin{thm}\label{tm103-2}
Let $1\leq p_0<q_0\leq \fz$ and $T$ be a bounded sublinear operator on $L^{p_0}(\rn, \ch_{\fz}^{\dz})$. Suppose there exist constants $\alpha_2>\alpha_1>1$ and $C>0$ such that for any cube $Q$, any $x\in Q$, and any $f\in L^{\fz}(\rn, \ch_{\fz}^{\dz})$ with compact support in $\rn\setminus\alpha_2 Q$,
\begin{align}\label{eq104-1}
\Bigl(\fint_Q|Tf(y)|^{q_0}\,d\ch_{\fz}^{\dz}\Bigr)^{\frac{1}{q_0}}\leq C\biggl\{\Bigl(\fint_{\alpha_1Q}|Tf(y)|^{p_0}\,d\ch_{\fz}^{\dz}\Bigr)^{\frac{1}{p_0}}+\cm_{\ch_{\fz}^{\dz}}(|f|^{p_0})^{\frac{1}{p_0}}(x)\biggr\}.
\end{align}
Then, whenever $p_0<p<q_0$ and $w\in \ca_{p/p_0, \dz}\cap \rh_{(q_0/p)',\dz}$, there exists a positive constant $C$ such that for all $f\in L^{\fz}(\rn, \ch_{\fz}^{\dz})$ with compact support,
\[
\|Tf\|_{L^p_{w}(\rn, \ch_{\fz}^{\dz})}\leq C\|f\|_{L^p_{w}(\rn, \ch_{\fz}^{\dz})}.
\]
\end{thm}

\begin{proof}
For any cube $Q\subset\rn$, define $\mathcal B_{l(Q)}f:=(1-\mathbf{1}_{\alpha_2Q})f$. Fix $f\in L^{\fz}(\rn, \ch_{\fz}^{\dz})$ with compact support, a cube $Q$, and $x\in Q$. Using the $L^{p_0}(\rn, \ch_{\fz}^{\dz})$ boundedness of $T$, we get
\begin{align}\label{eq104-2}
\Bigl(\fint_{\alpha_1Q}\bigl|T(I-\mathcal B_{l(Q)})f(y)\bigr|^{p_0}\,d\ch_{\fz}^{\dz}\Bigr)^{\frac{1}{p_0}}
&\ls \Bigl(\fint_{\alpha_2Q}\bigl|T(I-\mathcal B_{l(Q)})f(y)\bigr|^{p_0}\,d\ch_{\fz}^{\dz}\Bigr)^{\frac{1}{p_0}}\\
&\ls\cm_{\ch_{\fz}^{\dz}}(|f|^{p_0})^{\frac{1}{p_0}}(x)\noz.
\end{align}
In particular, \eqref{eq102-3} holds since $\alpha_1>1$.
In addition, it follows from \eqref{eq104-1}, sublinearity of $T$, and \eqref{eq104-2} that
\begin{align*}
&\Bigl(\fint_{Q}\bigl|T\mathcal B_{l(Q)}f(y)\bigr|^{q_0}\,d\ch_{\fz}^{\dz}\Bigr)^{\frac{1}{q_0}}\\
&\quad\ls\Bigl(\fint_{\alpha_1Q}\bigl|T\mathcal B_{l(Q)}f(y)\bigr|^{p_0}\,d\ch_{\fz}^{\dz}\Bigr)^{\frac{1}{p_0}}+\cm_{\ch_{\fz}^{\dz}}\bigl(|\mathcal B_{l(Q)}f|^{p_0}\bigr)^{\frac{1}{p_0}}(x)\\
&\quad\ls \Bigl(\fint_{\alpha_1Q}\bigl(|Tf(y)|^{p_0}+\bigl|T(I-\mathcal B_{l(Q)})f(y)\bigr|^{p_0}\bigr)\,d\ch_{\fz}^{\dz}\Bigr)^{\frac{1}{p_0}}+\cm_{\ch_{\fz}^{\dz}}(|f|^{p_0})^{\frac{1}{p_0}}(x)\\
&\quad\ls \cm_{\ch_{\fz}^{\dz}}(|Tf|^{p_0})^{\frac{1}{p_0}}(x)+\cm_{\ch_{\fz}^{\dz}}(|f|^{p_0})^{\frac{1}{p_0}}(x),
\end{align*}
which is \eqref{eq103-2} with $S_1=I$ and $S_2=0$. The conclusion follows by applying Theorem \ref{tm103-1} with $\mathcal X:=L^{\fz}(\rn, \ch_{\fz}^{\dz})$ with compact support and $\mathcal Y:=L^{p_0}(\rn, \ch_{\fz}^{\dz})$.
This completes the proof of Theorem \ref{tm103-2}.
\end{proof}
\medskip

\noindent\textbf{Acknowledgements}

This project is supported by the National Natural
Science Foundations of China (12201139), and the Natural Science Foundation of Hunan province (2024JJ3023).




\medskip

\noindent Long Huang
\medskip

\noindent School of Mathematics and Information Science,
Guangzhou University, Guangzhou, 510006, China

\smallskip

\noindent {\it E-mail}: \texttt{longhuang@gzhu.edu.cn}

\medskip
\medskip

\noindent Yangzhi Zhang (Corresponding author)
\medskip

\noindent School of Mathematics and Statistics, Central South University, Changsha, 410083, Hunan, China
\smallskip

\noindent {\it E-mail}:
\texttt{yzzhang@csu.edu.cn/ yzzhang@hunnu.edu.cn}

\medskip
\medskip

\noindent Ciqiang Zhuo
\smallskip
	
\noindent MOE-LCSM, School of Mathematics and Statistics,
Hunan Normal University,
Changsha, Hunan 410081, China
	
\smallskip

\noindent{\it E-mail:}
\texttt{cqzhuo87@hunnu.edu.cn}

\end{document}